\documentclass[reqno]{amsart}
\usepackage{amsmath,amsfonts,amssymb,amscd,multicol}
\usepackage[all]{xypic}
\usepackage[normalem]{ulem}
\usepackage{verbatim,color}

\usepackage{fullpage}

\DeclareMathOperator{\shHom}{\mathcal{H}\!\mathit{om}}

\newcommand{\on}{\operatorname}

\newcommand{\bpr}{\begin{proof}}
\newcommand{\epr}{\end{proof}}

\newcommand{\ol}{\overline}

\newcommand{\Llra}{\Longleftrightarrow}

\newcommand{\mc}{\mathcal}

\newcommand{\mb}{\mathbb}

\newcommand{\wt}{\widetilde}

\newcommand{\rcatMod}{\operatorname{Mod-}}

\newcommand{\rQch}{\operatorname{Qcoh}}

\newcommand{\rGr}{\operatorname{Gr-}\hskip -2pt}
\newcommand{\rQgr}{\operatorname{Qgr-}\hskip -2pt}

\newcommand{\rTors}{\operatorname{Tors-}\hskip -2pt}

\newcommand{\todo}[1]{{\color{red} #1 }}

\newcommand{\dirlim}{\underrightarrow{\lim}}

\newcommand{\wh}{\widehat}

\newcommand{\beq}{\begin{equation}}
\newcommand{\eeq}{\end{equation}}

\newcommand{\Hom}{{\rm Hom}}

\newcommand{\End}{{\rm End}}

\newcommand{\Ext}{{\rm Ext}}

\DeclareMathOperator{\im}{Im}

\newcommand{\coker}{\operatorname{coker}}

\numberwithin{equation}{section}
 \theoremstyle{plain}
\newtheorem{theorem}[equation]{Theorem}

\newtheorem{lemma}[equation]{Lemma}

\newtheorem{corollary}[equation]{Corollary}
\newtheorem{proposition}[equation]{Proposition}

\theoremstyle{definition}

\newtheorem{definition}[equation]{Definition}

\newtheorem{remark}[equation]{Remark}

\newtheorem{standing-hypothesis}[equation]{Standing Hypothesis}
\newtheorem{example}[equation]{Example}

\title{Well-closed subschemes of noncommutative schemes}
\author{D. Rogalski}
\date{}

\thanks{The author was partially supported by the NSF grant 
DMS-1201572 and the NSA grant H98230-15-1-0317.}

\subjclass[2010]{18E15, 
18A40, 
14A22. 
}

\keywords{Grothendieck category, noncommutative blowing up, adjoint functors, locally noetherian, closed subcategory}

\begin{document}

\begin{abstract}  
Van den Bergh has defined the blowup of a noncommutative surface at a point lying on a commutative divisor \cite{VdB}.  
We study one aspect of the construction, with an eventual aim of defining more general kinds of noncommutative blowups.
Our basic object of study is a quasi-scheme $X$ (a Grothendieck category).
Given a closed subcategory $Z$, in order to define a blowup of $X$ along $Z$ one first needs to have a functor $F_Z$ which is an 
analog of tensoring with the defining ideal of $Z$.  Following Van den Bergh, a closed subcategory $Z$ which has such a functor is 
called well-closed.  We show that well-closedness can be characterized by the existence of certain projective effacements for each object of $X$, and that the needed functor $F_Z$ has an explicit description in terms of such effacements.  As an application, we prove 
that closed points are well-closed in quite general quasi-schemes.
\end{abstract}

\maketitle

\section{Introduction}
This paper is the first part of a bigger project to study further the method of noncommutative blowing up developed by Van den Bergh in the monograph \cite{VdB}.  Van den Bergh gives a description of the blowup of a noncommutative surface at a point lying on a commutative divisor on the surface, and shows that this construction has many good properties, similar to those of a usual commutative blowup.  Our overall aim is both to make some aspects of Van den Bergh's procedure more explicit, as well as to show that the same ideas apply in a broader setting, which would allow one to define blowups of more general subschemes of noncommutative schemes.  In this paper, we focus primarily on a categorical notion which is a  fundamental building block for Van den Bergh's blowing up machinery:  the well-closedness of a closed subscheme of a noncommutative scheme.  Roughly speaking, this is a condition that one has a right exact functor which is an analog of 
``tensoring with the ideal sheaf defining the closed subscheme".  

Following \cite{VdB}, we take as our main object of study in noncommutative geometry a \emph{quasi-scheme}, in other words a Grothendieck category $X$.   A Grothendieck category is an abelian category with exact direct limits and a generator.  We review some background on such categories in Section~\ref{sec:background}.   We often assume in addition that $X$ is locally noetherian, in other 
words that $X$ has a set of noetherian generators.   Let $k$ be a field.  Important examples of quasi-schemes include $X = \rcatMod A$, the category of right modules over a noetherian $k$-algebra $A$, and $X = \rQch Y$, the category of quasi-coherent sheaves on a noetherian (commutative) $k$-scheme $Y$.  We are especially interested in applications to noncommutative projective geometry.  The most fundamental kind of quasi-scheme in this setting is $X = \rQgr A$, the quotient category of right $\mb{Z}$-graded modules over a connected $\mb{N}$-graded finitely generated noetherian $k$-algebra $A$, modulo the subcategory of modules which are direct limits of finite-dimensional modules.  
We call such an $X$ a \emph{noncommutative projective scheme}.  By a result of Serre, when $A$ is commutative and generated by its degree 1 elements, then $\rQgr A$ is equivalent to the category of quasi-coherent sheaves on the scheme $\on{Proj} A$.  Many important constructions in commutative algebraic geometry can be understood purely in terms of the category of quasi-coherent sheaves, and this justifies the study of Grothendieck categories as a replacement for schemes in the noncommutative case, where constructions involving actual spaces and sheaves on them are often not available.

For example, it is straightforward to find a reasonable definition of a closed subscheme $Z$ of a quasi-scheme $X$.  This is the 
notion of a \emph{closed subcategory}, namely, a full abelian subcategory $Z$ of $X$ such that the inclusion functor $i: Z \to X$ 
has both a left and right adjoint.  In most cases, this is equivalent to $Z$ being closed under subquotients, direct sums, and products.  
As evidence that this is the right definition, the closed subcategories of an affine quasi-scheme $X = \rcatMod A$ are precisely the categories $\rcatMod A/I$ for 2-sided ideals $I$ of $A$ \cite[Proposition 6.4.1]{Ros}, and the closed subcategories of a category $\rQch Y$ of quasi-coherent sheaves on a quasi-projective $k$-scheme $Y$ are the categories $\rQch W$ for closed subschemes $W$ of $Y$ \cite[Theorem 4.1]{Sm1}. 
Smith works out many basic properties of closed subcategories of quasi-schemes in \cite{Sm1}.  In particular, he defines 
the analogs of closed points in quasi-schemes $X$:  a \emph{closed point} is a closed subcategory $Z$ of $X$ which is equivalent 
to $\rcatMod D$ for some division ring $D$.  Assuming that $X$ is locally noetherian, this has a more intrinsic description as follows.  
 A simple object $P$ in a locally noetherian quasi-scheme $X$ is called \emph{tiny} if $\Hom_X(M, P)$ is a finitely generated $\End_X(P)$-module for all noetherian objects $M$.   When $P$ is tiny, the subcategory $Z$ consisting of all direct sums of $P$ is a closed point, where 
$Z \simeq \rcatMod D$ for $D = \End_X(P)$.  Conversely, all closed points are of this form \cite[Theorem 5.5]{Sm1}.

Since blowing up is such an important construction in commutative algebraic geometry, it is essential to have some analog of this in the noncommutative case.  Recall that if $Y$ is a commutative scheme with closed subscheme $W$ defined by a sheaf of ideals $\mc{I}$, 
then the blowup of $Y$ along $W$ is defined to be the relative Proj of a sheaf of graded algebras, 
namely $\on{Bl}_W Y = \on{\mathbf{Proj}}  (\mc{O}_Y \oplus \mc{I} \oplus \mc{I}^2 \oplus \dots)$.  
Trying to mimic this definition for more general quasi-schemes $X$, one runs into immediate problems.   Though we know 
what a closed subscheme of $X$ should be, finding replacements for the defining ideal sheaf of $Z$, what it means to take the powers of this ideal sheaf in order to define the Rees ring, and what the analog of the relative Proj construction should be, are major difficulties.

Van den Bergh's elegant solution is to work in a category of functors.   Let $Z$ be a closed subcategory of a quasi-scheme $X$.
While in general there is no object of $X$ that plays the role of an ideal sheaf defining $Z$, in nice cases one can find a right exact functor which plays the role of \emph{tensoring} with the ideal sheaf.  In more detail, there is a left exact functor $o_Z: X \to X$ 
which takes an object $M$ to the largest subobject of $M$ in $Z$.  Because Grothendieck categories have enough injectives, one can show that the category $\mc{L}(X, X)$ of left exact functors $X \to X$ is an abelian category, and so there is some left exact functor $G_Z: X \to X$ which fits into an exact sequence of left exact functors $0 \to o_Z \to o_X \to G_Z \to 0$, where $o_X: X \to X$ is the identity functor.   For example, 
when $X = \rcatMod A$ for a ring $A$, then $G_Z = \Hom(I, -)$, where $Z = \rcatMod A/I$.  Thus in general, we think of $G_Z$ 
as ``Hom from the ideal defining $Z$".   Now suppose that $G_Z$ has a left adjoint $F_Z: X \to X$.   For example, 
if $X = \rcatMod A$, then $F_Z = - \otimes I$.  Thus in general we think of $F_Z$ (when it exists) as the required analog of ``tensoring with the ideal defining $Z$".  Following Van den Bergh, we say that $Z$ is a \emph{well-closed} subcategory of $X$ if it is a closed 
subcategory and the functor $G_Z$ defined above has a left adjoint $F_Z$. 

Showing that $Z$ is well-closed is the first step in trying to define a blowup of a quasi-scheme $X$ along a closed subscheme $Z$.  
There are already a lot of interesting and nontrivial questions about this step.    Thus in this paper, our main goal is to study the property of well-closedness of subcategories $Z$ of quasi-schemes and the properties of the functors $F_Z$, in order 
to lay groundwork for a more general theory of blowing up.  
The two main questions we address in this paper are the following.  First, if $Z$ is well-closed, does the functor $F_Z$ have a more explicit description, other than just being the adjoint of $G_Z$?  Second, which closed subcategories $Z$ of quasi-schemes are well-closed?  To give more context to 
the first question,  in \cite{VdB}, when Van den Bergh shows that certain categories are well-closed, the argument relies ultimately on Freyd's adjoint functor theorem, which gives a criterion for when a left exact functor has a left adjoint.
The adjoint functor theorem is very abstract, and it seems difficult to see from its proof what the left adjoint it finds actually does to objects and morphisms.

We will give an answer to the first question above that works in wide generality.  
We first show in Lemma~\ref{lem:coker} below that the left exact functor $G_Z$ can be given the following explicit description.
If $i: Z \to X$ is the inclusion functor, its respective right and left adjoints $i^!: X \to Z$ and $i^*: X \to Z$ are given explicitly as follows:  
$i^!(N)$ is the unique largest subobject of $N$ which is in $Z$, and $i^*(N)$ is the unique largest factor object of $N$ which is in $Z$.
Now given an object $M \in X$, $G_Z(M) = \overline{M}/i^!(\overline{M})$, where $\overline{M}/M = i^!(E(M)/M)$ for an injective hull $E(M)$ of $M$.  Thus $G_Z$ first extends $M$ by the largest possible essential extension by an object in $Z$, and then mods out by the largest subobject in $Z$.   The action of $G_Z$ on morphisms can be defined similarly; any morphism $M \to N$ extends (non-uniquely) to a morphism $\overline{M} \to \overline{N}$, 
which induces a morphism of the factor objects $G_Z(M) \to G_Z(N)$ (which does not depend on the choice of extension). 

The left adjoint $F_Z$ of $G_Z$, when it exists, can be described in a roughly dual way to the explicit description of $G_Z$ just given, even though $X$ does not have enough projectives in general, much less projective covers.  Given a collection $S$ of objects in $X$, an \emph{$S$-projective effacement} of an object $M$ is an epimorphism $\pi: \underline{M} \to M$ satisfying the following lifting property:  given any epimorphism $f: P \to M$ with kernel in $S$, there exists $g: \underline{M} \to P$ such that $f g = \pi$.  We now give our first main result.
\begin{theorem} (Theorem~\ref{thm:equivs}.)
\label{thm:mainthm1} Let $X$ be a locally noetherian Grothendieck category, let $Z$ be a closed subcategory of $X$, and let $S$ be the collection of all objects in $Z$ which are injective in the category $Z$.  Let $0 \to o_Z \to o_X \to G_Z \to 0$ be the exact sequence in the category $\mc{L}(X, X)$ as above.  Then the following are equivalent:
\begin{enumerate}
\item $G_Z$ has a left adjoint $F_Z$; that is, $Z$ is well-closed in $X$.
\item The natural map $\Ext^1(M, \prod_{\alpha} N_{\alpha}) \to \prod_{\alpha} \Ext^1(M, N_{\alpha})$ is 
an isomorphism, for all small collections  $\{ N_{\alpha} \}$ of objects in $S$ and for all $M \in X$.
\item Every object $M$ in $X$ has an $S$-projective effacement.
\end{enumerate}
\end{theorem}
\noindent Moreover, when the conditions in the theorem hold, then $F_Z$ can be explicitly constructed.  On objects, $F_Z(M) = K_Z(\underline{M})$, where $\underline{M} \to M$ is a fixed $S$-projective effacement of $M$ and $K_Z$ is the functor which takes an object $N$ to its unique smallest subobject $N'$ such that $N/N' \in Z$.  To define the action of $F_Z$ on morphisms, given a morphism $M \to N$, it lifts (non-uniquely) to a morphism $\underline{M} \to \underline{N}$, which restricts to a morphism $G_Z(M) \to G_Z(N)$ (which does not depend on the choice of lift).  

Van den Bergh also defines a stronger condition on a closed subcategory $Z$ called \emph{very well-closed}, 
which is equivalent to $Z$ being well-closed and the category $Z$ having exact direct products \cite[Corollary 3.4.11]{VdB}.
Only very special closed subcategories should be expected to satisfy this stronger condition, 
but it is quite useful when it holds.   Our second main result is a characterization of very well-closedness similar to 
Theorem~\ref{thm:mainthm1}.  
\begin{theorem}  (Theorem~\ref{thm:equivs2}.)
\label{thm:mainthm2} Let $X$ be a locally noetherian Grothendieck category and let $Z$ be a closed subcategory of $X$.
Then the following are equivalent:
\begin{enumerate}
\item $Z$ is well-closed in $X$ and the category $Z$ has exact direct products; that is, $Z$ is very well-closed in $X$.
\item The natural map $\Ext^1(M, \prod_{\alpha} N_{\alpha}) \to \prod_{\alpha} \Ext^1(M, N_{\alpha})$ is 
an isomorphism, for all small collections  $\{ N_{\alpha} \}$ of objects in $Z$ and for all $M \in X$.
\item Every object $M$ in $X$ has a $Z$-projective effacement.
\end{enumerate}
\end{theorem}

Our description of the functors $F_Z$ using projective effacements is helpful in giving a partial answer to the second question, concerning which 
closed subcategories are well-closed, in the important special case of closed points.
\begin{theorem} (Theorem~\ref{thm:pt}.)
\label{thm:mainthm3}
Let $X$ be a locally noetherian Grothendieck $k$-category.  Suppose that $Z$ is a closed point in $X$, that is, $Z$ is the category of direct sums of a tiny simple object $P$ in $X$.  Suppose further that $\dim_k \Ext^1_X(M,P) < \infty$ for all noetherian objects $M \in X$.  
Then $Z$ is very well-closed in $X$.  
\end{theorem}
\noindent 
The necessary hypothesis that  $\dim_k \Ext^1_X(M,P) < \infty$ for noetherian objects $M$ is automatic in many cases of interest, in particular 
for nice noncommutative projective schemes.  For example,  if $A$ is a connected graded noetherian algebra satisfying the Artin-Zhang $\chi_1$ condition, $A$ is generated in degree $1$ as an algebra, and $M$ is a point module for $A$, then $P = \pi(M) \in X = \rQgr A$ is a tiny simple satisfying the hypothesis above, so the corresponding closed point $Z$ is very well-closed in $X$.  See Section~\ref{sec:examples} for more details.

The structure of the paper is as follows.  In Section~\ref{sec:background} we review the basic theory of Grothendieck categories and their closed subcategories.  In Section~\ref{sec:effacements}, we work in a general abelian category and study $S$-projective effacements, as well as the dual concept of $T$-injective effacements, and show how they may be used to define adjoint pairs of functors.  Then in Section~\ref{sec:eff-groth} we specialize to a Grothendieck category, and prove some results about projective effacements in that particular setting.  
Section~\ref{sec:wellclosed} contains the proofs of Theorem~\ref{thm:mainthm1} and Theorem~\ref{thm:mainthm2}.  Finally, 
in Section~\ref{sec:examples} we show how the theory of the paper works out in specific examples of quasi-schemes, especially noncommutative projective schemes, and give the proof of Theorem~\ref{thm:mainthm3}.

To conclude the introduction, we briefly describe our work in progress that continues the ideas of this paper.
First, we plan to apply the theory of this paper to define the blowing up of a quasi-scheme along a well-closed subcategory in general, and study some of its properties.  This blowup should generalize both Van den Bergh's blowups of surfaces at points on commutative divisors, as well as the naive blowups defined by Keeler, Stafford, and the author.   As a corollary, we plan to show carefully that the coordinate rings of Van den Bergh's del Pezzo surfaces in \cite{VdB} are the same as the subrings of the Sklyanin algebra defined in \cite{Ro}.  An interesting new case we plan to study is the blowup at a point of the category $\rQgr S$ of a 4-dimensional Sklyanin algebra $S$.   We also want to study in more detail the structure of closed subcategories of $\rGr A$ and $\rQgr A$, with hope that this will allow us to prove that more general closed subcategories of noncommutative projective schemes $\rQgr A$ are well-closed.

\section*{acknowledgments}

We thank Michel Van den Bergh and Paul Smith for helpful conversations.

\section{Categorical Preliminaries}
\label{sec:background}
In this section we review some of the basic category theory we need in this paper.  The reader 
looking for more extensive background about the foundations of abelian categories may consult \cite{Mac}.  Many of the standard facts about Grothendieck categories can be found in \cite{Ga} and \cite{Po}, and basic material on closed subcategories is given in \cite{Sm1}.

We adopt standard set-theoretic foundations for categories $X$ in terms of Grothendieck universes.  See, for example, \cite[Section I.6]{Mac}.   We assume a large enough universe $U$ (a set of sets), where the sets in $U$ are called \emph{small}, and assume that $\Hom_X(M,N)$ is a small set for any objects $M,N \in X$.  Direct sums and products are indexed only over small sets.  The universe $U$ is chosen so that given a union $T = \bigcup_{\alpha \in A} S_{\alpha}$, where the index set $A$ is small and each set $S_{\alpha}$ is small, then $T$ is small.   Also, the power set of a set in $U$ is again in $U$.  The set of all objects in the category $X$ is not small in general, however.   We say 
that $X$ is \emph{well-powered} if the set of subobjects of any given object in $X$ is small.  Occasionally we will wish to assume that $X$ is a \emph{$k$-category} for some field $k$: this means that all Hom-sets $\Hom_X(M,N)$ are $k$-vector spaces, and that composition is $k$-bilinear.

While our main results are proved in the setting of Grothendieck categories, as we define shortly, some of our results work in the setting of more general abelian categories $X$.   We will always assume at least that $X$ is an abelian category satisfying Grothendieck's axioms (AB3) and (AB3*), that is, that $X$ is \emph{cocomplete} (coproducts of small-indexed families of objects exist in $X$) and \emph{complete} (products of small-indexed families of objects exist in $X$), respectively.  
We use the notation $\bigoplus_{\alpha} M_{\alpha}$ and $\prod_{\alpha} M_{\alpha}$ for the respective coproduct and product of a small set of objects $M_{\alpha}$, and we also typically refer to coproducts as direct sums.

Let $Z$ be a full subcategory of a complete, cocomplete, and well-powered abelian category $X$, and let $i_Z = i_* : Z \to X$ be the inclusion functor.   We say that $Z$ is \emph{closed under subquotients} if every subobject and quotient object of an object in $Z$ is also in $Z$.     Following Smith \cite[Definition 2.4]{Sm1}, we say that $Z$ is \emph{weakly closed} if $Z$ is closed under subquotients and $i_Z$ has a right adjoint $i_Z^! = i^!: X \to Z$, and we say that $Z$ is \emph{closed} if $Z$ is closed under subquotients and $i_Z$ has both a right adjoint $i_Z^!$ and a left adjoint $i_Z^* = i^*: X \to Z$.  (Van den Bergh uses different terminology in \cite{VdB}, referring to weakly closed subcategories as closed (following Gabriel), and closed subcategories as biclosed.)  
It is elementary to see that if $Z$ is closed under subquotients, then $Z$ is weakly closed if and only if 
$Z$ is closed under direct sums, and $Z$ is closed if and only if $Z$ is closed under both direct sums and products \cite[Proposition 3.4.3]{VdB}.  Moreover, if $Z$ is weakly closed, then the right adjoint $i^!: X \to Z$ can be described as follows: given $M \in X$, the object $i^!(M)$ is the unique largest subobject of $X$ which is in $Z$, and $i^!$ acts on morphisms by restriction.  Similarly, if $Z$ is closed, then the left adjoint $i^*$ acts on objects as $i^*(M) = M/N$, where $N$ is the unique smallest subobject of $M$ such that $M/N \in Z$, and $i^*$ sends a morphism to the induced morphism of quotient objects.  By adjointness, $i^!$ is left exact and $i^*$ is right exact.  

Recall that a small set of objects $\{ O_{\alpha} \}$ is a \emph{set of generators} for $X$ if for all objects $N,P \in X$ and morphisms $f_1, f_2 \in \Hom_X(N, P)$ with $f_1 \neq f_2$, there exists a generator $O_{\alpha}$ and a map $g: O_{\alpha} \to N$ such that $f_1 g \neq f_2 g$.  If the set of generators consists of a single object $O$ then it is called a \emph{generator} for $X$.    Assuming $X$ is cocomplete, it is standard that $\{ O_{\alpha} \}$ is a set of generators for $X$ if and only if $O = \bigoplus_{\alpha} O_{\alpha}$ is a generator.   It is also easy to see when $X$ is cocomplete that $\{ O_{\alpha} \}$ is a set of generators if and only if for every $M \in X$, there is a epimorphism from some direct sum of copies of objects in the set of generators to $M$ \cite[Proposition 2.8.2]{Po}. Additionally, it is a standard fact that if $X$ has a generator then $X$ is well-powered \cite[Proposition I.5]{Ga}. 

An abelian category $X$ is a \emph{Grothendieck category} if it satisfies Grothendieck's axiom (AB5) (that is, it is cocomplete 
and direct limits of exact sequences are exact), and $X$ has a generator (or equivalently, a set of generators).   A 
Grothendieck category $X$ automatically has a number of useful other properties: it satisfies Grothendieck's axiom (AB4) 
that direct sums of exact sequences are exact (since this is known to be a general consequence of (AB5) \cite[Corollary 2.8.9]{Po}), it is 
well-powered (since it has a generator), and it is complete (Grothendieck's (AB3*) \cite[Corollary 3.7.10]{Po}).   Also, every object has an injective hull; in particular, the category has enough injectives \cite[Theorem 3.10.10]{Po}.  However, $X$ need not have enough projectives, and products of exact sequences need not be exact, that is, the category need not satisfy Grothendieck's axiom (AB4*).  
An important example of a locally noetherian Grothendieck category is the category $\rQch Y$ of quasi-coherent sheaves on a 
quasi-projective scheme $Y$, which typically does not have enough projectives or exact direct products.  We are especially interested 
in noncommutative projective schemes, as mentioned in the introduction; in the last section of the paper we will show how the theory we develop applies to these and other more specific examples.

As usual, an object in an abelian category is \emph{noetherian} if it has the ascending chain condition on subobjects.
The category $X$ is \emph{locally noetherian} if $X$ has a (small) set of noetherian generators.  We will assume in our  main theorems later that $X$ is a locally noetherian Grothendieck category.   Here are some other important properties of such categories.
\begin{lemma}
\label{lem:small}
 Let $X$ be a locally noetherian Grothendieck category with small set of noetherian generators $\{ O_{\alpha} \}$.
\begin{enumerate}
\item Direct limits and direct sums of injective objects in $X$ are injective, and every injective object is a direct sum of indecomposable injective objects.
\item Every indecomposable injective of $X$ is isomorphic to an injective hull of some epimorphic image of a generator $\mc{O}_{\alpha}$; in particular, the set of isomorphism classes of indecomposable injectives in $X$ is small.
\item If $Z$ is a closed subcategory of $X$ with inclusion functor $i_*: Z \to X$, then $Z$ is a locally noetherian Grothendieck category in its own right, with small set of noetherian generators  $\{ i^*(O_{\alpha}) \}$.
\end{enumerate}
\end{lemma}
\begin{proof}
(1).  See \cite[Theorem 5.8.7, Theorem 5.8.11]{Po}.   

(2).  Suppose that $I$ is indecomposable injective.  We may choose a nonzero map $f: O_{\alpha} \to I$ for some generator $\mc{O}_{\alpha}$.  Let $0 \neq M \subseteq I$ be the image of $f$.  Since $I$ is indecomposable injective, $I = E(M)$ must 
be an injective hull of $M$.   The last statement follows, since by well-poweredness the set of factor objects of a given object is small, and a small union of small sets is small.

(3).  The category $Z$ is clearly abelian since by definition it is closed under subquotients.   Since $Z$ is closed under direct sums and thus also under direct limits, it is clear that any direct limit in $X$ of objects in $Z$ is also a direct limit in the category $Z$.  Then since $X$ satisfies (AB5), so does $Z$.  Recall that $i^*(O_{\alpha})$ is the unique largest factor object of $O_{\alpha}$ in $Z$.  Given an object $M$ in $Z$ and a map $f: O_{\alpha} \to M$, the map factors through $i^*(O_{\alpha})$, and the statement about generators follows.
\end{proof}

Recall from the introduction that a \emph{closed point} of a locally noetherian Grothendieck category $X$ is a 
closed subcategory $Z$ which is equivalent to $\rcatMod D$ for some division ring $D$.  
We say that the locally noetherian Grothendieck category $X$ is \emph{locally finite} if $X$ has a set of generators which 
are both noetherian and artinian.  A closed point $Z$  is a locally finite Grothendieck category.
One way to get more locally finite categories is to join together closed points with the following construction.

\begin{definition}
Given a list of full subcategories $Z_1, Z_2, \dots, Z_n$ of an abelian category $X$, the \emph{Gabriel product} is the full subcategory $Z = Z_1 \cdot Z_2 \cdot \ldots \cdot Z_n$ consisting of objects $M$ with filtrations $0 = M_{n+1} \subseteq M_n \subseteq \dots \subseteq M_1 = M$ such that $M_i/M_{i+1} \in Z_i$ for all $1 \leq i \leq n$.
\end{definition}

\begin{lemma}  Let $X$ be complete, cocomplete, and well-powered.  If $Z_1, \dots, Z_n$ are closed subcategories of $X$, 
then $Z = Z_1 \cdot Z_2 \cdot \ldots \cdot Z_n$ is also closed.  
\end{lemma}
\begin{proof}
See \cite[Proposition 3.3.6]{VdB}.
\end{proof}

\begin{example}
Let $Z_1, Z_2, \dots, Z_n$ be a list of closed points, possibly with repeats, in the locally noetherian Grothendieck category $X$.
Then the category $Z = Z_1 \cdot Z_2 \cdot \ldots \cdot Z_n$ closed subcategory of $X$ by the previous lemma.  If $\{ O_{\alpha} \}$ 
is a set of noetherian generators for $X$, then $\{ i^*(O_{\alpha}) \}$ is a set of generators for $Z$, by Lemma~\ref{lem:small}.
It is clear that a noetherian object in $Z$ has finite length, so each $i^*(O_{\alpha})$ has finite length and $Z$ is a locally finite category.

Conversely, if $Z$ is a locally finite closed subcategory of a locally noetherian Grothendieck category $X$, and $Z$ is generated 
by finitely many objects of finite length whose composition factors are all tiny simples, then it is easy to see that $Z$ is a full subcategory of some Gabriel product $Z_1 \cdot Z_2 \cdot \ldots \cdot Z_n$ for some closed points $Z_i$.
\end{example}

In any Grothendieck category $X$, since $X$ has enough injectives, one has Ext groups $\Ext^i(M,N)$ for objects $M,N \in X$ and all $i \geq 0$, calculated using an injective resolution of $N$.   Since we assume that Hom sets are small, it is clear 
that $\Ext^i(M,N)$ is a small set for all $M,N$ and all $i \geq 0$.   There is also the usual correspondence between elements of $\Ext^1(M,N)$ and equivalence classes of short exact sequences $0 \to N \to P \to M \to 0$.   In the next result, we study how $\Ext$ interacts with directs sums and products in a Grothendieck category $X$.  Given a small index set $I$, let $\prod_I X$ be the 
category consisting of $I$-tuples of objects in $X$, and let $\prod: \prod_I X \to X$ be the functor given by 
$(M_{\alpha})_{\alpha \in I} \mapsto \prod_{\alpha \in I} M_{\alpha}$.  The functor $\prod$ is only left exact in general.  Since $X$ has enough injectives, so does $\prod_I X$, and thus we can define right derived functors $R^{i}\prod$ of the product functor.  
\begin{lemma}
\label{lem:extsumprod}
Let $X$ be a Grothendieck category.
\begin{enumerate}
\item Given $M$ and a family of objects $\{N_{\alpha} \}_{\alpha \in I}$ in $X$, the natural map 
 \[
\Ext^1(M, \prod N_{\alpha}) \overset{j}{\to} \prod \Ext^1(M, N_{\alpha})
\]
is an isomorphism if $R^1\prod(N_{\alpha}) = 0$.
\item Assume that $X$ is locally noetherian.  
Given a noetherian object $M$ and a directed system of objects $\{N_{\alpha} \}_{\alpha \in I}$ in $X$, for each $i \geq 0$ the natural map 
\[
\dirlim\ \Ext^i(M,N_{\alpha}) \to \Ext^i(M, \dirlim\ N_{\alpha}) 
\]
is an isomorphism.  Similarly, direct sums pull out of the second coordinate of $\Ext$.
\item Given a family of objects $\{ M_{\alpha} \}_{\alpha \in I}$ and $N$ in $X$, for each $i \geq 0$ there is an isomorphism
\[
\Ext^i(\bigoplus M_{\alpha}, N) \to \prod \Ext^i(M_{\alpha}, N).
\]
\end{enumerate}
\end{lemma}
\begin{proof}
(1)  Note that the natural map arises as follows:  For each $\beta$ there is a projection $\pi_{\beta}: \prod_{\alpha} N_{\alpha} \to N_{\beta}$ and thus by functoriality of $\Ext$ a map $f_{\beta}: \Ext^1(M, \prod N_{\alpha}) \to \Ext^1(M, N_{\beta})$.  By the  universal property of the product, the $f_{\beta}$ give a map  $\Ext^1(M, \prod N_{\alpha}) \to \prod \Ext^1(M, N_{\alpha})$ as required.

There is a Grothendieck spectral sequence
\[
E^{p, q}_2 = \Ext^p_X(M, \textstyle R^{q}\prod_{\alpha}  N_{\alpha}) \implies \prod_{\alpha} \Ext^{p+q}_X(M, N_{\alpha})
\]
associated to the composition of the product functor $\prod: \prod_I X \to X$ and the functor $\Hom_X(M, -)$ \cite[Equation (1.2)]{Roos}.  
The associated exact sequence of low degree terms begins
\begin{equation}
\label{eq:low}
\xymatrix{
0 \ar[r] & \Ext^1(M, \prod_{\alpha} N_{\alpha}) \ar^j[r]  & \prod_{\alpha} \Ext^1(M, N_{\alpha}) \ar[r]  & \Hom(M, R^1\prod N_{\alpha}) \ar[r] & \dots}
\end{equation}
We claim that the map $j$ in this sequence is the same as the natural map.  For fixed $\beta$, apply naturality of the exact sequence \eqref{eq:low} to the morphism $h: (N_{\alpha}) \to (M_{\alpha})$ in $\prod_I X$, where $M_{\beta} = N_{\beta}$, $M_{\alpha} = 0$ for $\alpha \neq \beta$, and $h_{\beta}$ is the identity map.  Since $\prod_{\alpha} M_{\alpha} \cong N_{\beta}$, one obtains a diagram
\[
\label{eq:low-nat}
\xymatrix{
0 \ar[r] & \Ext^1(M, \prod_{\alpha} N_{\alpha}) \ar^j[r] \ar^f[d]   & \prod_{\alpha} \Ext^1(M, N_{\alpha}) \ar[r] \ar^g[d] & \dots \\ 
0 \ar[r] & \Ext^1(M, N_{\beta}) \ar^1[r] & \Ext^1(M, N_{\beta}) \ar[r] & \dots }
\]
where $f$ is the map on $\Ext$ induced by the projection $\prod_{\alpha} N_{\alpha} \to N_{\beta}$, and $g$ is the projection onto 
the $\beta$th coordinate.  The commutation of this diagram for all $\beta$ implies that the map $j$ is the same as the natural map as claimed.
Now the result follows immediately from \eqref{eq:low}.

(2) For each $\alpha$, let $0 \to N_{\alpha} \to E^0_{\alpha} \to E^1_{\alpha} \to \dots$ be an injective resolution of $N_{\alpha}$.
Since direct limits of injectives are injective in $X$ by Lemma~\ref{lem:small}(1), and direct 
limits are exact in a Grothendieck category, 
$0 \to \dirlim\ N_{\alpha} \to \dirlim\ E^0_{\alpha} \to \dirlim\ E^1_{\alpha} \to \dots$ is an injective resolution of 
$\dirlim\ N_{\alpha}$.  Then $\Ext^i(M, \dirlim\ N_{\alpha})$ is the $i$th homology of the 
sequence $0 \to \Hom(M, \dirlim\ E^0_{\alpha}) \to \Hom(M, \dirlim\ E^1_{\alpha}) \to \dots$.
 Since $M$ is noetherian, the natural map $\dirlim\ \Hom(M, P_{\alpha}) \to \Hom(M, \dirlim\ P_{\alpha})$
is an isomorphism for any directed system $\{ P_{\alpha} \}$ \cite[Proposition 3.5.10, Exercise 5.8.2]{Po}.
Thus we can take the $i$th homology of 
the sequence $0 \to \dirlim\ \Hom(M, E^0_{\alpha}) \to \dirlim\ \Hom(M, E^1_{\alpha}) \to \dots.$
This is the same as $\dirlim\ \Ext^i(M, N_{\alpha})$ since direct limits are exact in the category of abelian groups.
Thus $\Ext^i(M, \dirlim\ N_{\alpha}) \cong \dirlim\ \Ext^i(M, N_{\alpha})$ for all $i$.

The proof that $\Ext^i(M, \bigoplus N_{\alpha}) \cong \bigoplus \Ext^i(M, N_{\alpha})$ is analogous, since direct sums are exact, and direct sums of injectives are also injective by Lemma~\ref{lem:small}(1).

(3)  This is an easy argument using the isomorphism $\Hom(\bigoplus M_{\alpha}, N) \to \prod \Hom(M_{\alpha}, N)$ and the fact that products are exact in the category of abelian groups. 
\end{proof}

\section{Defining functors via effacements}
\label{sec:effacements}

In this and the next section we define and study some categorical constructions which form the technical heart of the paper.
While our primary interest is in Grothendieck categories, it is more natural to first present the results in the context of general abelian categories.
Then beginning in Section~\ref{sec:eff-groth}, we will specialize to the case of Grothendieck categories $X$ and study the more special 
features of that setting.

Let $X$ be an abelian category which is cocomplete, complete, and well-powered.  
Let $Z$ be a closed subcategory of $X$, and let $i_* : Z \to X$ be the inclusion functor.    Recall that this means that $i_*$ has right adjoint $i^!: X \to Z$, where $i^!(M)$ is the unique largest subobject of $X$ which is in $Z$, and left adjoint $i^*: X \to Z$, where $i^*(M)$ is the unique largest factor object $M/N$ in $X$ such that $M/N \in Z$.    Let $o_Z = i_* i^*: X \to X$ and $o^Z = i_* i^!: X \to X$.  
In particular, $o_X = o^X: X \to X$ is notation for the identity functor on $X$.

It is standard that if $\mc{C}$ is any category, then since $X$ is abelian, the category $\on{Fun}(\mc{C}, X)$ of functors from $\mc{C}$ to $X$, with morphisms being natural transformations, is also an abelian category, with ``pointwise" operations.   (There is a minor set-theoretic issue that the Hom-sets in $\on{Fun}(\mc{C}, X)$ are not necessarily small, unless $\mc{C}$ is a small category, that is, the set of objects in $\mc{C}$ is small.   We will allow functor categories to have Hom-sets in a larger universe than the one we use for our categories $X$ of main interest.)  When $\mc{C}$ is an additive category, then the subcategory $\on{Add}(\mc{C}, X)$ consisting of additive functors from $\mc{C} \to X$ is an abelian subcategory of $\on{Fun}(\mc{C}, X)$, again with pointwise operations.   There is a morphism $\eta: o_X \to o_Z$ in 
$\on{Add}(X, X)$, where $\eta_M: M \to  i_* i^*(M) = M/N$ is the natural quotient map for each $M$.  Then $\eta$ is an epimorphism in $\on{Add}(X, X)$, and we let $K$ be its kernel.  Thus $K$ is an additive functor such that $K(M) = N$, where $N$ is the unique smallest subobject of $M$ such that $M/N \in Z$, and $K$ acts on morphisms by restriction.   Let $\mc{R}(X, X)$ be the category of right exact functors from $X \to X$.  In general $\mc{R}(X, X)$ is not abelian, and although $o_X$ and $o_Z$ are right exact, $K$ need not be right exact.   More specifically, if $0 \to N \overset{f}{\to} M \overset{g}{\to} P \to 0$ is a short exact sequence in $X$, then it is easy to check that $K(f)$ remains a monomorphism and $K(g)$ an epimorphism, but $0 \to K(N) \to K(M) \to K(P) \to 0$ need not be exact in the middle, having homology $(K(M) \cap f(N))/f(K(N))$ there.  

Dually, there is a monomorphism $\rho: o^Z \to o_X$ in $\on{Add}(X, X)$ where $\rho_M: i_*i^!(M) \to M$ is the natural 
inclusion, identifying $i^!(M)$ with the largest subobject of $M$ which is in $Z$.  Let $C$ be the cokernel of $\rho$ in $\on{Add}(X, X)$, so 
$C$ acts on objects by $C(M) = M/i^!(M)$.  Similarly, the category of left exact functors $\mc{L}(X, X)$ is not necessarily abelian, and the functor $C$ preserves monomorphisms and epimorphisms but need not belong to $\mc{L}(X, X)$, even though $o_X$ and $o^Z$ do.  Our aim is to define and study ``corrected" versions of the functors $K$ and $C$ which have better exactness properties and in some circumstances form an adjoint pair.

\begin{definition}  Let $S$ be any collection of objects in an abelian category $X$.  An epimorphism $\pi: \underline{M} \to M$
is an \emph{$S$-projective effacement} of $M$ if given any epimorphism $f: P \to M$ with kernel in $S$, there exists $g: \underline{M} \to P$ such that $f g = \pi$.  If every object $M$ in $X$ has an $S$-projective effacement, then we say that $X$ \emph{has $S$-projective effacements}.   Dually, an \emph{$S$-injective effacement} of $M$ is an injection $\iota: M \to \overline{M}$ such that given any monomorphism $h: M \to Q$ with cokernel in $S$, there is $j: Q \to \overline{M}$ such that $j h =\iota$; we say that $X$ \emph{has $S$-injective effacements} if every object in $X$ has an $S$-injective effacement.
\end{definition}
\noindent  In this definition, we have anglicized the terms ``effacement projectif/injectif" which we believe 
were originally due to Grothendieck.  We have also added the dependence on $S$ (usually effacements are defined in the case $S = X$). 

Note that the kernel of an $S$-projective effacement of $M$ (or the cokernel of an $S$-injective effacement) is not required to be in $S$, though in the applications it will be convenient to have this additional property.  We have the following easy observation in the important special case that $S$ consists of objects in a closed subcategory $Z$ of $X$.
\begin{lemma}
\label{lem:kerinS}
Let $Z$ be a closed subcategory of the abelian category $X$, and let $S$ be some collection of objects in $Z$.
If $M \in X$ has an $S$-projective effacement, then $M$ has an $S$-projective effacement $\pi_M: \underline{M} \to M$ such that $\ker \pi_M \in Z$.  
Similarly, if $M$ has a $S$-injective effacement 
then $M$ has a $S$-injective effacement with cokernel in $Z$.
\end{lemma}
\begin{proof}
Let $M' \to M$ be an $S$-projective effacement, and let $0 \to L \to M' \to M \to 0$ be the corresponding exact sequence.   Now $i^*(L) = L/J$ is the largest factor object of $L$ in $Z$, the sequence  
\[
0 \to K = L/J \to \underline{M} = M'/J \to M \to 0
\]
is exact, and $\underline{M} \to M$ is easily seen to be an $S$-projective effacement, whose kernel $K$ is now in $Z$.   The statement for injective effacements is proved dually.
\end{proof}

Given epimorphisms $\pi_M: M \to M'$ and $\pi_N: N \to N'$, we say that a morphism $f: M \to N$ \emph{lifts} the morphism $g: M' \to N'$ if $g \circ \pi_M = \pi_N \circ f$.  Dually, given monomorphisms $\iota_M: M' \to M$ and $\iota_N: N' \to N$, we say that a morphism $f: M \to N$ 
\emph{extends} the morphism $g: M' \to N'$ if $ f \circ \iota_M = \iota_N \circ g$.  
Projective effacements have a general morphism lifting property (and injective effacements have a morphism extending property) which we give in the next result.  
\begin{lemma}
\label{lem:liftextend}
\begin{enumerate}
\item
Let $\pi_M: \underline{M} \to M$ be an $S$-projective effacement, and let $\pi_N: N' \to N$ be any epimorphism with kernel in $S$.  Given 
any morphism $f: M \to N$, there is a morphism $f': \underline{M} \to N'$ lifting $f$.  
\item  Let $\iota_N: N \to \overline{N}$ be a $T$-injective effacement, and let $\iota_M: M \to M'$ be any monomorphism with cokernel in $T$.  
Given any morphism $g: M \to N$, there is a morphism $g': M' \to \overline{N}$ extending $g$.
\end{enumerate}
\end{lemma}
\begin{proof}
We prove only the first statement, since the second is proved dually.
We construct the commutative diagram
\[
\xymatrix{ &  & \underline{M} \ar[r]^{\pi_M} \ar[d]^g & M  \ar[d]^{=} &  \\
0 \ar[r] & L \ar[r]\ar[d]^{=} & P \ar[r]\ar[d]^h & M \ar[r]^q \ar[d]^f & 0 \\
0 \ar[r] & L \ar[r] & N' \ar[r]^{\pi_N} & N \ar[r] & 0 }
\]
as follows.
The bottom row is given, with $L = \ker \pi_N \in S$ by assumption.  The second row is formed by letting $P$ 
be the pullback of the maps $f$ and $\pi_N$; see \cite[Lemma 7.29]{Rot}.  The map $g$ exists since $\pi_M$ is 
an $S$-projective effacement.  Then $h g = f'$ is the required lift of $f$.
\end{proof}

We now show how to use projective and injective effacements to construct some important functors related to a closed subcategory of $X$.
\begin{proposition}
\label{prop:functors}
Let $X$ be a cocomplete, complete, and well-powered abelian category.  
Let $Z$ be a closed subcategory of $X$, and let the functors $i^!, i^*, C, K: X \to X$ be defined as above with respect to $Z$. 
\begin{enumerate}
\item Let $S$ be any collection of objects of $Z$.  For each $M \in X$, assume that there exists an $S$-projective effacement $\pi_M: \underline{M} \to M$ with kernel in $S$, and fix one such.   Define $F(M) = K(\underline{M})$.  For any morphism $f: M \to N$ in $X$, by Lemma~\ref{lem:liftextend} there exists a morphism $\underline{f}: \underline{M} \to \underline{N}$ lifting $f$.  Fix any such $\underline{f}$, and apply $K$ to give $F(f) = K(\underline{f}): F(M) \to F(N)$.  Then $F: X \to X$ is a functor which is independent up to natural isomorphism of the choices of projective effacements and lifts.  There is a canonical morphism of functors $\nu: F \to o_X$.
\item Let $T$ be any collection of objects of $Z$.   For each $M \in X$, assume that there exists an $T$-injective effacement $\iota_M: M \to \overline{M}$ which has cokernel in $T$, and fix one such.   Define $G(M) = C(\overline{M})$.  For any morphism $f: M \to N$ in $X$, by Lemma~\ref{lem:liftextend} there exists a morphism $\overline{f}: \overline{M} \to \overline{N}$ extending $f$.  Fix any such $\overline{f}$, and apply $C$ to give $G(f) = C(\overline{f}): G(M) \to G(N)$.  Then $G: X \to X$ is a functor which is independent up to natural isomorphism of the choices of injective effacements and extensions.  There is a canonical morphism of 
functors $\mu: o_X \to G$.
\item Suppose that both parts (1) and (2) apply, and that for each $M \in X$ we have 
$i^!(\overline{M}) \in S$ and $i^*(\underline{M}) \in T$.  Then the functors $(F,G)$ form an adjoint pair.
In particular, this holds if $S$ and $T$ are both equal to all objects in $Z$.  
\end{enumerate}
\end{proposition}
\begin{proof}
(1)  First we show that $F(f)$ does not depend on the choice of lift $\underline{f}$.  Suppose that we have two different lifts $g_1, g_2: \underline{M} \to \underline{N}$ of $f$.  Then $\pi_N \circ (g_1-g_2) = 0$, and so in particular $\im (g_1 -g_2) \cong \underline{M}/\ker (g_1 - g_2) \in Z$.  Since $F(M)$ is the smallest subobject $M'$ of $\underline{M}$ such that $\underline{M}/M' \in Z$, we see that $F(M) \subseteq \ker (g_1 - g_2)$, and $g_1 \big\vert_{F(M)} =   g_2 \big\vert_{F(M)}$.  Thus $K(g_1) = K(g_2)$, and so $F(f)$ is independent of the choice of lift.   The independence of the choice of lifts also easily implies that $F$ is a functor:   given $f: M \to N$ and $h: N \to P$, we first choose lifts $\underline{f}: \underline{M} \to \underline{N}$ and $\underline{h}: \underline{N} \to \underline{P}$, and then we may choose $\underline{h} \circ \underline{f}: \underline{M} \to \underline{P}$ as our lift of $h \circ f$,
from which follows 
\[
F(h \circ f) = K(\underline{h} \circ \underline{f}) = K(\underline{h}) \circ K(\underline{f}) = F(h) \circ F(f),
\]
since $K$ is a functor.   Similarly, choosing the identity map $1_{\overline{M}}: \overline{M} \to \overline{M}$ as our lift of the indentity map $1_M: M \to M$ implies that $F(1_M) = K(1_{\overline{M}}) = 1_{F(M)}$.  

Next we show that the definition of $F$ is independent (up to natural isomorphism) of the arbitrary choices of projective effacements.  Suppose that for each object $M$ we choose an $S$-projective effacement $\pi'_M: M' \to M$ (also with kernel in $S$), and use these to define a functor $F'$ in the same way.  By Lemma~\ref{lem:liftextend}, there are maps $g: \underline{M} \to M'$ and $h: M' \to \underline{M}$ which lift the identity map $1_M$.  Then $g \circ h: M' \to M'$ is a lift of the identity map $1_{M'}$; since the identity map $1_{M'}: M' \to M'$ is also a lift, by the independence of lifts proved above, $K(g) \circ K(h) = K(g \circ h) = K(1_{M'}) = 1_{F'(M)}$ is the identity.  Similarly, $K(h) \circ K(g) = K(h \circ g)  = K(1_{\underline{M}}) = 1_{F(M)}$.  Thus $\eta_M = K(g): F(M) \to F'(M)$ is an isomorphism for each $M$.  To see that $\eta$ is natural, if $f: M \to N$ is a morphism, choose lifts $\underline{f}: \underline{M} \to \underline{N}$ and $f': M' \to N'$ of $f$ by Lemma~\ref{lem:liftextend}.  Choose $g_1: \underline{M} \to M'$ lifting $1_M$ and $g_2: \underline{N} \to N'$ lifting $1_N$ as above.  Then $f' \circ g_1$ and $g_2 \circ \underline{f}: \underline{M} \to N'$ both lift $f$.   By the independence 
of lifts, 
\[
F'(f) \circ \eta_M = K(f') \circ K(g_1) = K(f' \circ g_1) = K(g_2 \circ \underline{f}) = K(g_2) \circ K(\underline{f}) = \eta_N \circ F(f),
\]
as required.   Thus $\eta: F \to F'$ is a natural isomorphism of functors.

Finally, for each object $M$ there is a morphism $\nu_M: F(M) \to M$ given by 
$\nu_M = \pi_M \big\vert_{K(\underline{M})}$.  The maps $\nu_M$ are natural, since 
if $f: M \to N$ then we have $f \circ \pi_M = \pi_N \circ \underline{f}$ by construction, and $K$ is a functor.
Thus we have a morphism of functors $\nu: F \to o_X$.

(2)  This is completely dual to part (1), so the proof is omitted.

(3)  For any fixed objects $M, N \in X$ consider the diagram
\[
\xymatrix{ \Hom(FM, N)  & \Hom(\underline{M}, \overline{N}) \ar[l]_{\phi} \ar[r]^{\psi}  & \Hom(M, GN),
 }
\]
where the maps $\phi$ and $\psi$ are defined as follows.   Given $h \in \Hom(\underline{M}, \overline{N})$, applying $K$ gives 
a map $K(h): F(M) = K(\underline{M}) \to K(\overline{N})$.  Use the injection $\iota_N: N \to \overline{N}$ to identify $N$ with 
a subobject of $\overline{N}$.  Since $\overline{N}/N \in Z$, we have $K(\overline{N}) \subseteq N$ because $K(\overline{N})$ is 
the smallest such subobject.  Thus $\on{im}(K(h)) \subseteq N$ and so $K(h)$ defines a morphism $\phi(h): F(M) \to N$.  Similarly, applying $C$ to $h$ gives $C(h): C(\underline{M}) \to C(\overline{N}) = G(N)$.  Since the kernel of $\pi_M: \underline{M} \to M$ is in $Z$, there is an epimorphism $M \to C(\underline{M})$ and $C(h)$ induces a map $\psi(h): M \to G(N)$.

There is an epimorphism $\overline{N} \to G(N)$ whose kernel $i^!(\overline{N})$ is in $S$ by hypothesis.  Thus by
Lemma~\ref{lem:liftextend}, given $g \in \Hom(M, GN)$, there 
is $\underline{g} \in \Hom(\underline{M}, \overline{N})$ lifting $g$.  This is equivalent to $\psi(\underline{g}) = g$.  Thus $\psi$ is surjective.  A dual argument using the assumption that $i^*(\underline{M}) \in T$ shows that $\phi$ is surjective.

Now suppose that $h \in \ker(\phi)$.  This means that $h(F(M)) = 0$, and in particular, $\im(h)$ is isomorphic to a factor 
object of $\underline{M}/F(M)$.  Thus $\im(h) \in Z$.  Conversely, if $\im(h) \in Z$, then $\underline{M}/\ker(h) \in Z$ and so 
$F(M) \subseteq \ker(h)$ since $F(M)$ is the smallest such subobject.  In conclusion, $\ker(\phi)$ consists of those $h$ whose 
image is in $Z$.  A similar argument shows that $\ker(\psi)$ has the same description, so $\ker(\phi) = \ker(\psi)$.  Thus there are induced bijections 
$\wt{\phi}: \Hom(\underline{M}, \overline{N})/(\ker \phi) \to \Hom(FM, N)$ and 
$\wt{\psi}: \Hom(\underline{M}, \overline{N})/(\ker \psi) \to \Hom(M, GN)$ and 
$\eta_{M,N}: \wt{\psi} \circ \wt{\phi}^{-1}: \Hom(FM, N) \to \Hom(M, GN)$ is a bijection.

A diagram chase similar to the arguments already given shows that these isomorphisms $\eta_{M,N}$ are natural in $M$ and $N$, so $(F, G)$ is an adjoint pair as claimed.  
\end{proof}

When the category $X$ has enough injectives, the constructions above involving injective effacements can be described in 
a much simpler way; dually, if the category has enough projectives then all of the results above using projective effacements become simplified.  We make this precise in the next results.

Recall that $\mc{L}(X, X)$ denotes the category of left exact functors $X \to X$, with morphisms given by natural transformations.  Suppose that $X$ has enough injectives.   Then it is standard that $\mc{L}(X, X)$ is an abelian category, in the following way.  There is an equivalence of categories $\gamma: \operatorname{Add}(\operatorname{Inj}(X), X) \to \mc{L}(X, X)$ where $\operatorname{Inj}(X)$ is the full subcategory of $X$ consisting of injective objects \cite[Proposition 3.1.1(3)]{VdB}.  Explicitly, given an additive functor $G': \operatorname{Inj}(X) \to X$, for each $M \in X$ one takes the beginning of an injective resolution $M \to E_0 \to E_1 \to \dots$ in $X$ and defines $G(M) = \ker  (G'(E_0) \to G'(E_1))$.  Also, a morphism $f: M \to N$ induces a morphism $G(M) \to G(N)$ by taking any lift of $f$ to a morphism of the injective resolutions, applying $G'$, and taking the induced map of the kernels.  It is straightforward to check that this defines a left exact functor $G$ 
which is independent up to natural isomorphism of the choices involved.  As we noted in the previous section, $\on{Add}(\operatorname{Inj}(X), X)$ is automatically abelian by taking objectwise kernels and cokernels.  Thus $\mc{L}(X, X)$ is also an abelian category via the equivalence $\gamma$.  It is easy to see that in the category $\mc{L}(X, X)$, kernels are still computed objectwise, but cokernels are not, in general.  Analogously, if $X$ has enough projectives, the the category $\mc{R}(X, X)$ of right exact functors $X \to X$ is abelian, since it is equivalent to the category $\operatorname{Add}(\on{Pro}(X), X)$; cokernels in $\mc{R}(X, X)$ are computed objectwise, but not kernels in general.

\begin{lemma}
\label{lem:coker}
Let $X$ be an abelian category which is complete, cocomplete, and well-powered.  Let $Z$ be a closed subcategory of $X$.  Suppose that $X$ has enough injectives; in particular $X$ automatically has $Z$-injective effacements, and for each $M \in X$ 
we can fix such an effacement $M \to \overline{M}$ with cokernel in $Z$, by Lemma~\ref{lem:kerinS}.
Let $G = G_Z$ be the functor defined by Proposition~\ref{prop:functors}(2) with $T = Z$, and let $\mu = \mu_Z: o_X \to G$ be the corresponding morphism.
Then there is an exact sequence
\[
0 \to o_Z \overset{\theta}{\to} o_X \overset{\mu}{\to} G \to 0
\]
in $\mc{L}(X, X)$, where $\theta: o_Z \to o_X$ is the natural morphism.
\end{lemma}
\begin{proof}
The natural transformation $\theta$ is a monomorphism in $\mc{L}(X, X)$ because $\theta_M: o_Z(M) \to o_X(M)$ is a monomorphism in $X$ for all objects $M$.  Let $\mu': o_X \to G'$ be the  cokernel of $\theta$ in $\mc{L}(X, X)$.  By the constructions 
outlined above, a left exact functor is uniquely determined by its restriction to the full subcategory of injective objects.  For an injective $I$ we have $G'(I) = \coker (i^!(I) \to I) = I/i^!(I)$.  Since $I$ is injective, it is clear that the 
identity map $1_I: I \to I$ is a $Z$-projective effacement of $I$.  Thus by the definition of the functor $G$, we have 
$G(I) = C(I) = G'(I)$, where $C$ is defined as in the previous section, and by construction the morphism $\mu: o_X \to G$ 
defined in Proposition~\ref{prop:functors}(2) is also given on $I$  by the natural map $I \to I/i^!(I)$.  
Thus $\mu = \mu'$, and $\mu$ is a cokernel of $\theta$ as required.
\end{proof}
\noindent 
The previous result shows that if $X$ has enough injectives, to construct the functor $G_Z$ for any closed subcategory 
$Z$ we can simply take a cokernel of $\theta_Z: o_Z \to o_X$ in $\mc{L}(X, X)$.  It is not necessarily to use injective effacements at all.

For completeness, we state without proof the dual result we get when the category $X$ has enough projectives.
\begin{lemma}
\label{lem:ker}
Let $X$ be an abelian category which is complete, cocomplete, and well-powered.  Let $Z$ be a closed subcategory of $X$.  Suppose that $X$ has enough projectives; in particular $X$ automatically has $Z$-projective effacements, and for each $M \in X$ 
we can fix such an effacement $\underline{M} \to M$ with cokernel in $Z$, by Lemma~\ref{lem:kerinS}.
Let $F = F_Z$ be the functor defined by Proposition~\ref{prop:functors}(2), and let $\nu = \nu_Z: F \to o_X$ be the corresponding morphism.
Then there is an exact sequence
\[
0 \to F \overset{\nu}{\to} o_X \overset{\rho}{\to} o^Z \to 0
\]
in $\mc{R}(X, X)$, where $\rho: o_X \to o^Z$ is the natural morphism. 
\end{lemma}

\section{Effacements in Grothendieck categories}
\label{sec:eff-groth}

We now specialize the theory of the previous section to a Grothendieck category $X$ with a closed subcategory $Z$. 
As usual, let $i_*: Z \to X$ be the inclusion functor and $i^!: X \to Z$, $i^*: X \to Z$ its respective right and left adjoints, 
and let $o_Z = i_* i^!, o^Z = i_* i^*: X \to X$.  A Grothendieck category $X$ has injective hulls, and an injective hull $M \to E(M)$ is a $T$-injective effacement for any collection of objects $T$.  Thus the constructions using injective effacements in the previous sections can be described in the simple way provided by Lemma~\ref{lem:coker}.  In particular, there is a left exact functor $G_Z$
defined by Proposition~\ref{prop:functors}(2) with $T = Z$, and $G_Z$ is also the cokernel of the natural map $\theta: o_Z \to o_X$ in the abelian category $\mc{L}(X, X)$ of left exact functors.   It is important to know when $G_Z$ has a left adjoint.  Since many Grothendieck categories of interest do not have enough projectives, the idea is to use Proposition~\ref{prop:functors} to construct a left adjoint $F_Z$ to the functor $G_Z$.  For this, we need $X$ to have $S$-projective effacements for a suitable collection of objects $S$ in $Z$.

First, we will need to study some more general properties of $S$-projective effacements.
We have the following alternative characterization of an $S$-projective effacement of $M$.
\begin{lemma}
\label{lem:pe-equiv}
Let $X$ be a Grothendieck category, and let $S$ be any collection of objects in $X$.  The following are equivalent for an object $M \in X$ and a surjection $\pi: \underline{M} \to M$:
\begin{enumerate}
\item[(i)] The morphism $\pi$ is an $S$-projective effacement.
\item[(ii)] For every $N \in S$, the map $\pi^*: \Ext^1_X(M, N) \to \Ext^1_X(\underline{M}, N)$ induced by $\pi$ is $0$.
\end{enumerate}
\end{lemma}
\begin{proof}
This is an immediate generalization of \cite[Corollary 1.4]{Roos}, which proves the result in case $S = X$.  
\end{proof}

Next, we see that in certain cases, to show that $X$ has $S$-projective effacements it suffices to check the definition for only some 
objects in $X$ or only some objects in $S$.
\begin{lemma}
\label{lem:pe-props}  Fix a collection $S$ of objects in the Grothendieck category $X$.
\begin{enumerate}
\item If $M \in X$ has an $S$-projective effacement, then so does any epimorphic image of $M$.
\item If $\{ M_{\alpha} \}$ is a set of objects, each of which has an $S$-projective effacement $\pi_{\alpha}: \underline{M_{\alpha}} \to M_{\alpha}$, then the direct sum $\oplus \pi_{\alpha}: \bigoplus_{\alpha} \underline{M_{\alpha}} \to \bigoplus_{\alpha} M_{\alpha}$ is an $S$-projective effacement of $\bigoplus_{\alpha} M_{\alpha}$.
\item If $\{ O_{\alpha} \}$ is a (small) set of generators for $X$, then $X$ has $S$-projective effacements if and only if every $O_{\alpha}$ has an $S$-projective effacement. 
\item Suppose that $X$ is locally noetherian.  Let $S'$ be a collection of objects such that every object in $S$ is a direct sum or a direct limit of objects in $S'$.   If $\pi: \underline{M} \to M$ is a $S'$-projective effacement for a noetherian object $M \in X$, then $\pi$ is also an $S$-projective effacement.
\end{enumerate}
\end{lemma}
\begin{proof}
(1)  Let $f: M \to N$ be an epimorphism, and suppose that $\pi: \underline{M} \to M$ is an $S$-projective effacement of $M$.  If 
$p: Q \to N$ is an epimorphism with kernel in $S$, then Lemma~\ref{lem:liftextend} shows that there is is a map $\wt{f}: \underline{M} \to Q$ covering $f$, in other words such that $p \wt{f} = f \pi$.  This shows that $f \pi: \underline{M} \to N$ is an $S$-projective effacement of $N$.

(2) Use Lemma~\ref{lem:pe-equiv} and the following commutative diagram, in which the vertical maps are isomorphisms by Lemma~\ref{lem:extsumprod}(3):
\[
\xymatrix{\prod_{\alpha} \Ext^1(M_{\alpha}, N) \ar[r] \ar^{\cong}[d] & \prod_{\alpha} \Ext^1(\underline{M_{\alpha}}, N) \ar^{\cong}[d] \\
\Ext^1(\bigoplus_{\alpha} M_{\alpha}, N) \ar[r]  & \Ext^1(\bigoplus_{\alpha} \underline{M_{\alpha}}, N).  \\
}
\]

(3) This follows from (1) and (2), since every object in $X$ is an epimorphic image of a direct sum of generators.

(4) Consider the following commutative diagram for an arbitrary direct sum $\bigoplus N_{\alpha}$ with each $N_{\alpha} \in S'$:
\[
\xymatrix{\bigoplus_{\alpha} \Ext^1(M, N_{\alpha}) \ar^{0}[r] \ar^{\cong}[d] & \bigoplus_{\alpha} \Ext^1(\overline{M}, N_{\alpha}) \ar[d] \\
\Ext^1(M, \bigoplus_{\alpha} N_{\alpha}) \ar[r]  & \Ext^1(\overline{M}, \bigoplus_{\alpha} N_{\alpha}).  \\
}
\]
The top arrow is $0$ since $\ol{M} \to M$ is an $S'$-projective effacement.  The left vertical arrow is an isomorphism since $M$ is a noetherian object, using Lemma~\ref{lem:extsumprod}(2).  It follows that the bottom arrow is also $0$.   Given a directed system $\{ N_{\alpha} \}$ instead, the same argument works to show 
that the map $\Ext^1_X(M, \dirlim\ N_{\alpha}) \to \Ext^1(\ol{M}, \dirlim\ N_{\alpha})$ is zero, again using Lemma~\ref{lem:extsumprod}(2).  
Since every object in $S$ is a direct sum or direct limit of objects in $S'$, 
we conclude that $\overline{M} \to M$ is an $S$-projective effacement by Lemma~\ref{lem:pe-equiv}.
\end{proof}

When applying Proposition~\ref{prop:functors} to construct a functor using $S$-projective effacements, it is necessary to have effacements whose kernels 
are actually in $S$.  This is easy when $S$ is equal to all objects in the closed subcategory $Z$, by Lemma~\ref{lem:kerinS}.  The following lemma gives another 
important case where we can guarantee this.
\begin{lemma}
\label{lem:goodeff}
Let $Z$ be a closed subcategory of the Grothendieck category $X$.  Let $S$ be the collection of objects in $Z$ which are injective in the category $Z$.  
If $M \in X$ has an $S$-projective effacement, then there is an $S$-projective effacement $\pi_M: \underline{M} \to M$ such that $\ker \pi_M \in S$.
\end{lemma}
\begin{proof}
By Lemma~\ref{lem:kerinS}, since $M$ has an $S$-projective effacement, we can find an $S$-projective effacement 
$t: M' \to M$ with $K = \ker t \in Z$.  Now let $g: K \to E = E_Z(K)$ be an injective hull in the category $Z$.   If $P$ is the pushout of the maps $f: K \to M'$ and $g$, we obtain an exact sequence $0 \to E \to P \to M \to 0$, giving the first two rows of the commutative diagram below \cite[Lemma 7.28]{Rot}.  

We claim that $\pi: P \to M$ is an $S$-projective effacement; then we will be done since its kernel $E$ is apparently in $S$.   Now if $0 \to F \overset{r}{\to} Q \overset{s}{\to} M \to 0$ is any short exact sequence with $F \in S$, consider the following commutative diagram: 
\[
\xymatrix{0 \ar[r] & E  \ar^b[r] &  P  \ar^{\pi}[r] & M \ar[r] & 0 \\
0 \ar[r] & K \ar_{h'}[d] \ar^g[u] \ar^f[r] & M' \ar^{i}[u]\ar_h[d] \ar^{t}[r] & M \ar^{=}[u]\ar_{=}[d] \ar[r] & 0  \\
0 \ar[r] & F \ar^r[r] & Q \ar^s[r] & M \ar[r] & 0.
}
\]
Here, the map $h$ exists completing the bottom right square since $t: M' \to M$ is an $S$-projective effacement, and the map $h'$ is induced by $h$.
Since $F$ is in $S$ and so is injective in $Z$, there is $d: E \to F$ such that $dg = h'$.   Since $rdg = rh' = hf$, by the universal property of the pushout
there is a morphism $j: P \to Q$ such that $h = ji$ and $rd = jb$.   Then $\pi i = t = sh = sji$, and so $(\pi - sj) i = 0$.   Now by construction of the pushout, 
one has $\im(i) + \im(b) = P$.  We have seen that $\im(i) \subseteq \ker (\pi - sj)$.  We also have  $\pi b = 0$ and $sjb = srd = 0$, and then 
$(\pi - sj) b = \pi b - sjb = 0$, so that $\im(b) \subseteq \ker (\pi - sj)$.  Thus $P \subseteq \ker (\pi - sj)$ and 
$\pi = sj$, which shows that $\pi$ is an $S$-projective effacement as required.
\end{proof}

Recall that since $X$ has $Z$-injective effacements, for each $M \in X$ we can choose a $Z$-injective effacement with cokernel 
in $Z$, by Lemma~\ref{lem:kerinS}.  Following the proof of that lemma, we see in fact that for each $M$ we have a canonical (up to isomorphism) $Z$-injective effacement $j: M \to \overline{M}$, with $\overline{M}/M = i^!(E(M)/M)$,  where $E(M)$ is the injective hull of $M$.  In other words, 
$\overline{M}$ is the maximal essential extension of $M$ by an object in $Z$.  These canonical $Z$-injective effacements have the following property.
\begin{lemma}
\label{lem:inj-bottom} Let $Z$ be a closed subcategory of the Grothendieck category $X$.   For $M \in X$, 
let $j: M \to \overline{M}$ be the canonical $Z$-injective effacement discussed above.  Then applying the functor $i^!$ yields an injection $i^!(j): i^!(M) \to i^!(\overline{M})$, which is an injective hull of $i^!(M)$ in the category $Z$.
\end{lemma}
\begin{proof}
Let $N = i^!(M)$, in other words the largest subobject of $M$ in $Z$.  Let $N \subseteq I$ be an injective hull in the category $Z$.  Then $N \subseteq I$ is essential in 
the category $X$ as well, and so we can choose an injective hull $E(N)$ of $N$ in $X$ with $N \subseteq I \subseteq E(N)$.   Since $N \subseteq M$, we can choose an injective hull $E(M)$ of $M$ in $X$ so that $E(N) \subseteq E(M)$.  

Working in $E(M)$, $I+M/M \cong I/(I \cap M)$ is in  $Z$, and thus $I+M \subseteq \overline{M}$ since $\overline{M}/M$ is the largest subobject of $E(M)/M$ in $Z$.  In particular, $I \subseteq \overline{M}$ and thus $I \subseteq i^!(\overline{M})$.  On the other hand, if $0 \neq P \subseteq i^!(\overline{M})$, then $P \cap M \neq 0$ as $M \subseteq \overline{M}$ is essential.  Since $P \in Z$, $0 \neq P \cap M \subseteq  i^!(M)$ and this proves that $i^!(j): i^!(M) \to i^!(\overline{M})$ is essential (in $X$ or in $Z$).  This forces $i^!(\overline{M}) = I$ since $I$ is a maximal essential extension in $Z$. 
\end{proof}

Considering the definition, an $S$-projective effacement of an object $M$ is some object which, speaking loosely, ought to be formed by sticking objects in $S$ to the bottom of $M$ in all possible nontrivial ways.  This naive description can actually be made formal in some cases, as we see in the proof of the following 
result.
\begin{proposition}
\label{prop:efface}  Let $X$ be a locally noetherian Grothendieck category, and let $S$ be a collection of objects in $X$ with a small subcollection $S' \subseteq S$ such that every element of $S$ is either a direct sum or a direct limit of objects in $S'$.

Fix a noetherian object $M \in X$, and suppose that for any small-indexed family of objects $(N_{\alpha})$ with each $N_{\alpha} \in S'$, the natural map
\begin{equation}
\label{eq:ext}
 \Ext^1_X(M, \prod_{\alpha} N_{\alpha}) \to \prod_{\alpha} \Ext^1_X(M, N_{\alpha})
\end{equation} is an isomorphism.  Then $M$ has an $S$-projective effacement.
\end{proposition}
\begin{proof}
We prove this first in the special case that $X$ is a $k$-category for a field $k$.  In this case we can construct 
a smaller projective effacement, which is useful in applications.  Afterward we indicate the easy changes necessary to 
remove the assumption that $X$ is a $k$-category.

Let $M$ be a noetherian object satisfying \eqref{eq:ext} for all small-indexed families of objects in $S'$.  For each $N \in S'$, the underlying set of
$\Ext^1_X(M,N)$ is small, so the union $\bigcup_{N \in S'} \Ext^1_X(M,N)$ is a small union of small sets and is thus small.
For each $N \in S'$, pick a $k$-basis $\beta_N$ for $\Ext^1_X(M, N)$, and let
$\beta = \bigcup_{N \in S'} \beta_N$ be the disjoint union of these bases, which is again a small set.  For each $v \in \beta$, let $N_v$ be a copy of the object $N$ such that $v \in \beta_N$.

Now we have an isomorphism
\begin{equation}
\label{eq:extprod}
\phi: \Ext^1_X(M, \prod_{v \in \beta} N_v) \to \prod_{v \in \beta} \Ext^1_X(M, N_v)
\end{equation}
by hypothesis.  
Note that there is a special element  of the right hand side given by $\theta = \prod_{v \in \beta} v$.  The element 
$\theta' = \phi^{-1}(\theta)$ of the left hand side represents some extension
\[
0 \to \prod_{v \in \beta} N_v \overset{i}{\to} \underline{M} \overset{\pi}{\to} M \to 0
\]
and we claim that $\pi$ is a $S'$-projective effacement of $M$.

Consider an extension $0 \to N \to P \overset{f}{\to} M \to 0$ with $N \in S'$ and the corresponding element $\rho \in \Ext^1_X(M, N)$.  We may write $\rho = \sum_{i = 1}^m a_i v_i$ with $a_i \in k$ and $v_i \in \beta_N \subseteq \beta$ for all $i$.  There is a morphism
\[
j: \prod_{v \in \beta} N_v \overset{q}{\to} \prod_{i=1}^m N_{v_i} \cong \bigoplus_{i=1}^m N \overset{p}{\to} N
\]
where $q$ is the natural projection map, and $p$ is given by the formula $\sum_{i=1}^m a_i 1_N$. 
Since finite direct sums pull out of the second coordinate of Ext (for instance, by Lemma~\ref{lem:extsumprod}(2)), 
the morphism $j$ induces a corresponding map of Ext groups
\[
\wh{j}: \Ext^1_X(M,  \prod_{v \in \beta} N_v) \to \Ext^1_X(M, \prod_{i=1}^m  N_{v_i}) \cong \bigoplus_{i=1}^m  \Ext^1_X(M,  N) \overset{\wh{p}}{\to} \Ext^1_X(M, N)
\]
where $\wh{p}$ is given by the formula $(w_1, \dots, w_m) \mapsto \sum_{i=1}^m a_i w_i$.  By construction, $\wh{j}(\theta') = \rho$.  This means that there is 
a commutative diagram
\[
 \xymatrix{
0 \ar[r] & \prod_{v \in \beta} N_v \ar^i[r] \ar^j[d] & \underline{M} \ar^{\pi}[r] \ar^h[d] & M \ar[r] \ar^{=}[d] & 0 \\
0 \ar[r] & N \ar[r] & P \ar^f[r] & M \ar[r] & 0},
\]
where $P$ is a pushout of $i$ and $j$ \cite[Formula II, p. 429]{Rot}.
Then $h: \underline{M} \to P$ satisfies $f h = \pi$, proving the claim that $\pi$ is a $S'$-projective effacement.

Now by Lemma~\ref{lem:pe-props}(4), since $M$ is noetherian and $\pi: \underline{M} \to M$ is a $S'$-projective effacement, it is also an $S$-projective effacement. 

When $X$ is not necessarily a $k$-category, essentially the same proof works, with the following changes.  First, one replaces $\beta_N$ with the entire 
set $\Ext^1_X(M,N)$ instead, so that $\beta = \bigcup_{N \in S'} \Ext^1_X(M,N)$, which is still small.   Again one writes $N_v$ for a copy of  the $N$ such that $v \in \Ext^1_X(M,N)$.  Given an extension $\rho \in \Ext^1_X(M,N) \subseteq \beta$, one replaces the map $j$ above simply 
with the projection onto the single $\rho$th coordinate:  $j: \prod_{v \in \beta} N_v \to N_{\rho} = N$.  The rest of the proof is the same.
\end{proof}

\begin{remark}
\label{rem:single}
We note that only the single instance \eqref{eq:extprod} of the hypothesis \eqref{eq:ext} of the proposition is used in the proof.   Thus in practice one 
only needs to verify the single equation \eqref{eq:extprod} in order for the result to hold.
\end{remark}

\section{Well-closed and very well-closed subcategories}
\label{sec:wellclosed}

We are now ready to prove our main theorems.  

\begin{theorem}
\label{thm:equivs} Let $X$ be a locally noetherian Grothendieck category, let $Z$ be a closed subcategory of $X$, and let $S$ be the collection of all objects in $Z$ which are injective in the category $Z$.  Let $0 \to o_Z \to o_X \to G \to 0$ be the exact sequence in $\mc{L}(X, X)$ of Lemma~\ref{lem:coker}.

Then the following are equivalent:
\begin{enumerate}
\item $G$ has a left adjoint $F$.
\item The functor $G$ commutes with products of objects in $X$.
\item $[R^1 \prod](N_{\alpha}) = 0$ for all small families $\{ N_{\alpha} \}$ of objects in $S$. 
\item The natural map $\Ext^1(M, \prod_{\alpha} N_{\alpha}) \to \prod_{\alpha} \Ext^1(M, N_{\alpha})$ is 
an isomorphism, for all small families  $\{ N_{\alpha} \}$ of objects in $S$ and for all $M \in X$.
\item The category $X$ has $S$-projective effacements.



\end{enumerate}
\end{theorem}
\begin{proof}
Van den Bergh studies condition $(3)$ and shows that $(1) \Llra (3)$ \cite[Proposition 3.4.7]{VdB}.  Also, the 
equivalence $(1) \Llra (2)$ follows from Freyd's adjoint functor theorem \cite[Theorem 2.1(1)]{VdB}.

Lemma~\ref{lem:extsumprod}(1) shows that $(3) \implies (4)$.

Now suppose that condition $(4)$ holds.   Let $S'$ be the set of isomorphism classes of indecomposable injective objects in the category $Z$.  By Lemma~\ref{lem:small}, 
$Z$ is also a locally noetherian Grothendieck category, every injective object in $Z$ is a direct sum of indecomposable injectives in $Z$, 
and the set $S'$ is small.    Thus by Proposition~\ref{prop:efface}, if $M$ is a 
noetherian object in $X$, then $M$ has an $S$-projective effacement.   Since $X$ has a set of noetherian generators, it 
follows that every $M \in X$ has an $S$-projective effacement by Lemma~\ref{lem:pe-props}(3).   So $(4) \implies (5)$.

Now assuming $(5)$, then in fact every $M \in X$ has an $S$-projective effacement $\pi_M: \underline{M} \to M$ with $\ker \pi_M \in S$, by Lemma~\ref{lem:goodeff}.  Thus Proposition~\ref{prop:functors}(1) applies and constructs a functor $F$.   Let $T$ be the class of all objects in $Z$.  
We already have the canonical $T$-injective effacement $\iota_N: N \to \overline{N}$ with cokernel in $T$, where $\overline{N}/N = i^!(E(N)/N)$, as described 
before Lemma~\ref{lem:inj-bottom}, and Proposition~\ref{prop:functors}(2) constructs a functor which is the same as $G$ up to natural isomorphism, 
by Lemma~\ref{lem:coker}.  we have $i^*(\underline{M}) \in T$ trivially, and $i^!(\overline{N}) \in S$ 
holds for all $N \in X$ by Lemma~\ref{lem:inj-bottom}.   So Proposition~\ref{prop:functors}(3) applies and shows that $(F, G)$ form an adjoint pair.  Thus $(5) \implies (1)$.
\end{proof}

\begin{definition}
Let $Z$ be a closed subcategory of a locally noetherian Grothendieck category $X$.   
We say that $Z$ is \emph{well-closed} in $X$ if any of the  equivalent conditions of Theorem~\ref{thm:equivs} holds.
The definition follows Van den Bergh, who uses this term for condition (3) in the theorem \cite[Definition 3.4.6]{VdB}.
\end{definition}

\noindent 
We expect well-closedness to be a very general condition.  In fact, we do not know any example of a closed subset of a locally noetherian Grothendieck category that is not well-closed.    For some applications it is useful to study a more special condition on closed subcategories than well-closedness.  We have the following variant of Theorem~\ref{thm:equivs}.
\begin{theorem}
\label{thm:equivs2} Let $X$ be a locally noetherian Grothendieck category and let $Z$ be a closed subcategory of $X$.
Then the following are equivalent:
\begin{enumerate}
\item $[R^1 \prod](N_{\alpha}) = 0$ for all small families $\{ N_{\alpha} \}$ of objects in $Z$.
\item $Z$ is well-closed in $X$ and the category $Z$ has exact direct products.
\item The natural map $\Ext^1(M, \prod_{\alpha} N_{\alpha}) \to \prod_{\alpha} \Ext^1(M, N_{\alpha})$ is 
an isomorphism, for all small families  $\{ N_{\alpha} \}$ of objects in $Z$ and for all $M \in X$.
\item Every object $M$ in $X$ has a $Z$-projective effacement.
\end{enumerate}
\end{theorem}
\begin{proof}
Van den Bergh proves that $(1) \Llra (2)$ \cite[Corollary 3.4.11]{VdB}.  

Lemma~\ref{lem:extsumprod}(1) shows that $(1) \implies (3)$.   

Assume now that $(3)$ holds.  Let $S = Z$ and let $S'$ be the collection of isomorphism classes of noetherian objects in $Z$.  The category $Z$ is itself Grothendieck, 
and so has a generator $O$.  It is easy to see that any noetherian object in $Z$ is an epimorphic image of $O^{\oplus n}$ for some $n$.  By well-poweredness, 
the set $S'$ is a countable union of small sets and is thus small.  Because $Z$ is locally noetherian, every object in $Z$ is a direct limit of noetherian objects \cite[Proposition 8.6]{Po}.  Since every object in $S$ is a direct limit of objects in $S'$, Proposition~\ref{prop:efface} shows that every noetherian object $M$ in $X$ has an $S$-projective effacement.
Then every object in $X$ has an $S$-projective effacement by Lemma~\ref{lem:pe-props}(3).  Thus $(3) \implies (4)$.

Finally, assume that $(4)$ holds.  Then in particular, every $M \in X$ has an $S$-projective effacement, where $S$ is the collection of objects which are injective in the category $Z$.  Thus $Z$ is well-closed in $X$ by Theorem~\ref{thm:equivs}.  Suppose that $M \in Z$, and that $p: M' \to M$ is a $Z$-projective effacement of $M$ in $X$.   We can assume that $\ker p \in Z$ by Lemma~\ref{lem:kerinS}, but apriori, $M'$ need not be in $Z$.    
Let $\underline{M} = i^*(M') = M'/N$ be the largest factor object of $M'$ which is in $Z$.  Since $M \in Z$, $N \subseteq \ker p$ and so there is an induced epimorphism $\pi: \underline{M} \to M$.

Suppose that $0 \to K \to L \overset{f}{\to} M \to 0$ is a short exact sequence in the category $Z$.   There is $h: M' \to L$ such that $fh = p$, by the definition of $Z$-projective effacement.  Since $L \in Z$, $\im h \in Z$ and so $N \subseteq \ker h$.  In other words, $h$ factors through $\underline{M}$.  Thus $\pi: \underline{M} \to M$ is a $Z$-projective effacement in the category $Z$.  We conclude that the Grothendieck category $Z$ has $Z$-projective effacements.  By a result of Roos \cite[Theorem 1.3]{Roos}, this implies that $Z$ is an (AB4*) category, that is, that $Z$ has exact products.  Thus $(4) \implies (2)$.
\end{proof}

\begin{remark}
\label{rem:roosproof}
Roos leaves the proof of the just-cited part of \cite[Theorem 1.3]{Roos} to the reader.   The proof was not obvious to us, so 
we indicate here a possible argument that if a Grothendieck category $Z$ has $Z$-projective effacements, then it has exact products.   
Given the $Z$-projective effacement $\pi_M: \underline{M} \to M$  and a collection $\{ N_{\alpha} \}$ of objects in $Z$, recall the Grothendieck spectral sequence
\[
E^{p, q}_2 = \Ext^p_Z(M, \textstyle R^{q}\prod_{\alpha}  N_{\alpha}) \implies \prod_{\alpha} \Ext^{p+q}_Z(M, N_{\alpha})
\]
which was used in Lemma~\ref{lem:extsumprod}, now applied in the category $Z$, so that $R^{q}\prod$ is the qth right derived functor of the product functor in $Z$.  By naturality of the exact sequences of low degree terms, applied to the morphism $\pi_M$, we get the following commutative diagram: 
\[
\label{low-eq}
\xymatrix{
0 \ar[r] & \Ext^1_Z(M, \prod_{\alpha} N_{\alpha}) \ar^{d_1}[r]\ar^{g_0}[d]  & \prod_{\alpha} \Ext^1_Z(M, N_{\alpha}) \ar^{d_2}[r]\ar^{g_1}[d]  & \Hom_Z(M, R^1\prod N_{\alpha}) \ar^{d_3}[r]\ar^{g_2}[d] & \Ext^2_Z(M, \prod_{\alpha} N_{\alpha}) \ar^{g_3}[d]  \\
0 \ar[r] & \Ext^1_Z(\underline{M}, \prod_{\alpha} N_{\alpha}) \ar^{d'_1}[r]  & \prod_{\alpha} \Ext^1_Z(\underline{M}, N_{\alpha}) \ar^{d'_2}[r]  & \Hom_Z(\underline{M}, R^1\prod N_{\alpha}) \ar^{d'_3}[r] & \Ext^2_Z(\underline{M}, \prod_{\alpha} N_{\alpha}) }
\]
Since $\pi: \underline{M} \to M$ is a $Z$-projective effacement, $g_0 = g_1 = 0$ by Lemma~\ref{lem:pe-equiv}.  By left exactness of $\Hom$, $g_2$ is a monomorphism.  Since 
$g_2 d_2 =  d'_2 g_1 = 0$, we see that $d_2 = 0$.  Thus the natural map $\Ext^1_Z(M, \prod_{\alpha} N_{\alpha}) \to \prod_{\alpha} \Ext^1_Z(M, N_{\alpha})$ is an isomorphism.  Since $M$ was an arbitrary object in $Z$, this fact also applies to the object $\underline{M}$, which shows that $d'_2 = 0$.  Thus $d'_3$ is a monomorphism and hence so 
is $d'_3 g_2$.  Now let $E$ be an injective hull of  $\prod_{\alpha} N_{\alpha}$ in $Z$ and consider the short exact sequence 
$0 \to \prod_{\alpha} N_{\alpha} \to E \to C \to 0$.  Then $\Ext^2_Z(M, \prod_{\alpha} N_{\alpha}) \cong \Ext^1_Z(M, C)$ for all objects $M$; 
since the natural map $\Ext^1_Z(M, C) \to \Ext^1(\underline{M}, C) = 0$ because $\pi$ is a $Z$-projective effacement, we get $g_3 = 0$.
Now $d'_3 g_2 = g_3 d_3 = 0$, but we already saw that $d'_3 g_2$ is a monomorphism.  This forces $\Hom_Z(M, R^1\prod N_{\alpha}) = 0$.  Since this holds for all objects $M \in Z$, 
$R^1\prod N_{\alpha} = 0$.  This holds for all small families $\{ N_{\alpha} \}$ of objects in $Z$, so $Z$ has exact products.
\end{remark}

\begin{remark}
Theorem~\ref{thm:equivs2} also gives an alternative proof of the other (harder) direction of \cite[Theorem 1.3]{Roos}, that if a Grothendieck category $Z$ has exact direct products, then it has $Z$-projective effacements, but only in the special case that $Z$ is locally noetherian.
This follows from the proof that $(4) \implies (2)$ in Theorem~\ref{thm:equivs2}, in the case $X =Z$ (note that $Z$ is trivially well-closed in $Z$).
\end{remark}

\begin{definition}
Let $Z$ be a closed subcategory of a locally noetherian Grothendieck category $X$.   
We say that $Z$ is \emph{very well-closed} in $X$ if any of the equivalent conditions in Theorem~\ref{thm:equivs2} holds 
(Van den Bergh uses this term for condition (1)).
\end{definition}
\noindent Very well-closedness is clearly a much more special condition than well-closedness.
One of its advantages is that it is stable under Gabriel product.
\begin{lemma}
\label{lem:prodvwc}
Let $X$ be a locally noetherian Grothendieck category with closed subcategories $Z_1, Z_2$.  If $Z_1$ and $Z_2$ are very-well-closed in $X$, 
then so is $Z_3 = Z_1 \cdot Z_2$.
\end{lemma}
\begin{proof}
This is \cite[Proposition 3.5.12]{VdB}.  It is also easy to see why this is true in terms of projective effacements, as follows.  
For each $M \in X$ fix a $Z_1$-projective effacement $\pi_M: \underline{M} \to M$ and a $Z_2$-projective effacement $\rho_M: \uwave{M} \to M$.  We can assume that $\ker \pi_M \in Z_1$ and $\ker \rho_M \in Z_2$, by Lemma~\ref{lem:kerinS}.  Then $\theta_M = \pi_M \circ \rho_{\underline{M}}: \uwave{\underline{M}} \to M$ is a $Z_3 = Z_1 \cdot Z_2$-projective effacement of $M$, 
as is easy to check.
\end{proof}
\noindent There is no obvious reason, on the other hand, for the Gabriel product of well-closed subcategories to be well-closed, 
as Van den Bergh has also noted.  We do not know a counterexample, however.

One easy way to ensure (very) well-closedness, which often occurs in applications, is the following.
\begin{definition}
\label{def:self-effacing}
Let $Z$ be a closed subcategory of a locally noetherian Grothendieck category $X$, and let $S$ be a collection of objects in the 
category $Z$.  We say that an object $M \in X$ is \emph{$S$-projective self-effacing} if the identity map $M \to M$ is an $S$-projective effacement.
We say that $X$ has a \emph{set of $S$-projective self-effacing generators} if there is a small set of generators $\{ O_{\alpha} \}$ 
for $X$, where each $O_{\alpha}$ is $S$-projective self-effacing.
\end{definition}

\begin{proposition}
\label{prop:self-effacing}
\label{prop:cn}
Let $X$ be a Grothendieck category with closed subcategory $Z$, and let  $S$ be a collection of objects 
in the category $Z$.  Let $\{ O_{\alpha} \}$ be a small set of generators for $X$.
\begin{enumerate}
\item If $\Ext^1(O_{\alpha}, N) = 0$ for all $O_{\alpha}$ and all $N \in S$, 
then $\{ O_{\alpha} \}$ is a set of $S$-self-effacing generators for $X$.   

\item If $\{ O_{\alpha} \}$ is a set of $S$-self-effacing generators for $X$, then 
$X$ has $S$-projective effacements. 
\end{enumerate}
\end{proposition}
\begin{proof}
(1)  Given any short exact sequence $0 \to N \to P \to O_{\alpha} \to 0$ where $N \in S$, the sequence must be split since 
$\Ext^1(O_{\alpha}, N) = 0$.  Then it is clear that the identity map $O_{\alpha} \to O_{\alpha}$ is already  
an $S$-projective effacement, so each $O_{\alpha}$ is $S$-projective self-effacing. 

(2)  In particular, the hypothesis implies that each generator has an $S$-projective effacement.  Then $X$ has $S$-projective effacements, 
by Lemma~\ref{lem:pe-props}(3).  
\end{proof}
\noindent  For example, let $X$ be a locally noetherian Grothendieck category with noetherian generators $\{ O_{\alpha} \}$, let $Z$ be a closed subcategory, and let $S$ be the set of injective objects in the category $Z$.   If $\Ext^1(O_{\alpha}, E) = 0$ for all $E \in S$ and all 
$\alpha$, then the previous result implies that $Z$ is well-closed in $X$.  
We note that in this case we not only get that the functor 
$F = F_Z$ exists, but also that it can be described in a way similar to its description if $X$ were to have enough projectives, as given in Lemma~\ref{lem:ker}.  Namely, let $M \in X$.  Then there is a ``partial resolution" 
\[
P_1 \to P_0 \to M \to 0
\]
where $P_1, P_0$ are direct sums of generators.  The $P_i$ here are of course not projective, in general.  Still, we may apply $F$ to this sequence to obtain an exact sequence $F(P_1) \to F(P_0) \to F(M) \to 0$, since $F$ is right exact.  Since the $O_{\alpha}$ are $S$-projective self-effacing, 
the same is true for each $P$, by Lemma~\ref{lem:extsumprod}(3).  Thus by the construction of the functor $F$ given in Proposition~\ref{prop:functors}, we have $F(P_i) = K(P_i)$, where recall that $K$ is the functor that sends an object $N$ to its smallest subobject $N'$ such that $N/N' \in Z$.  Thus $K(P_1) \to K(P_0) \to F(M) \to 0$ is exact.  We see that 
$F(M)$ may be defined by taking such a resolution $P_1 \to P_0$ of $M$ by self-effacing objects, applying $K$, and taking the cokernel.
This is the same description as we would get if $X$ had enough projectives, where we would take a partial projective resolution instead.
However, defining the action of $F$ on morphisms is more awkward, as one lacks a comparison lemma for these resolutions 
by self-effacing objects.

\section{Examples}
\label{sec:examples}

We close the paper with some examples of how the theory works out for some important kinds of Grothendieck categories.
\begin{example}
\label{ex:ring-scheme}
The prototypical example of a Grothendieck category is the category $\rcatMod R$ of right modules over a ring $R$.   The closed subcategories of $\rcatMod R$ are exactly the subcategories of the form $\rcatMod R/I$ for (two-sided) ideals $I$ in $R$, by a result of Rosenberg \cite[Proposition 6.4.1]{Ros}.  The category has the generator $R$ and if $R$ is right noetherian, then $\rcatMod R$ is locally noetherian.

The category $X$ has enough projectives and exact products, so it is obvious that every closed subcategory $Z = \rcatMod R/I$ is very well-closed.  In fact $G = G_Z$ has the explicit description $G = \Hom_R( I, -)$ and the functor $F = F_Z$ has the explicit description $F = - \otimes_R I$.
Also, $\{R \}$ is a projective generator, so it is $S$-projective self-effacing for any collection of objects $S$.
\end{example}

\begin{example}
A somewhat less trivial example is the category $X = \rQch Y$ of quasi-coherent sheaves on a $k$-scheme $Y$.   For simplicity, 
suppose that $Y$ is projective over $k$ in this example.  Then it is well-known that $X$ is a Grothendieck category.  If $\mc{L}$ is an ample invertible sheaf on $Y$, then $\{ \mc{L}^{\otimes n} | n \in \mb{Z} \}$ is a set of noetherian generators for $X$, so $X$ is certainly locally noetherian.
Smith has shown that the closed subcategories of $X$ are those of the form $Z = \rQch W$ for closed subschemes $W$ of $Y$ \cite[Theorem 4.1]{Sm1}.  
The category $X$ has enough injectives, but not enough projectives, in general, and need not have exact products.  (We are unaware of a general result about which schemes have categories of quasi-coherent sheaves with exact direct products, but $\mb{P}^1$ already does not.)

If $\mc{I}$ is the ideal sheaf defining the closed subscheme $W$ then the functor $G_Z$ has the explicit description $G_Z =  \shHom_{\mc{O}_X}(\mc{I}, -)$, and this has the obvious left adjoint $F_Z = (-) \otimes_{\mc{O}_X} \mc{I}$.  Thus every closed subcategory $Z$ is well-closed, but need not be very well-closed (since this is equivalent to $Z$ having exact products).  

Suppose that $\mc{E}$ is an injective object in the category $Z = \rQch W$.  Then $\mc{E}$ is flasque on $W$, so it is also flasque (but 
not injective in general) when considered as a sheaf on $Y$.  This is enough to conclude that $H^1(Y, \mc{E}) 
= \Ext^1_X(\mc{O}_X, \mc{E}) = 0$ \cite[Proposition 2.5]{Ha}.   Similarly, 
$\Ext^1(\mc{L}^{\otimes n}, \mc{E}) \cong \Ext^1(\mc{O}_X, \mc{E} \otimes \mc{L}^{\otimes -n}) = 0$ since 
$\mc{E} \otimes \mc{L}^{\otimes -n}$ is still flasque.  In particular, $X$ has a set of $S$-projective self-effacing generators, where $S$ is the class of injective objects in $Z$.
\end{example}
\noindent We assumed projectivity of $Y$ in the previous example for convenience only.  Ryo Kanda has proved that for any locally noetherian scheme $Y$, the closed subcategories of $\rQch Y$ are still exactly the categories $\rQch W$ for closed subschemes $W$ of $Y$ \cite[Theorem 11.11]{Kan}.  

\begin{example}
let $H$ be a group and let $R = \bigoplus_{h \in H} R_h$ be an $H$-graded $k$-algebra.  Let $X = \rGr R$ be the category of $H$-graded 
right $R$-modules.  Then $X$ is a Grothendieck category.   For any $h \in H$ and $M \in X$ we have the shifted module 
$M(h)$ with $M(h)_g = M_{hg}$.  Then $\{ R(h) | h \in H \}$ is a set of generators for $X$, and so if $R$ is graded right noetherian, then 
$X$ is a locally noetherian category.  In fact the $R(h)$ are projective generators, so $X$ has enough projectives.  The category $X$ also clearly has 
exact products.  Because of this, every closed subcategory $Z$ of $X$ must be well-closed, using the characterization in Theorem~\ref{thm:equivs}(3).
Also, $Z$ clearly inherits the property of having exact direct products, so in fact $Z$ is very well-closed.  Since the $R(h)$ are projective, 
they are in fact $S$-projective self-effacing generators for $X$, for any collection $S$ of objects.

If $I$ is a graded ideal of $R$, then $Z = \rGr R/I$ is closed in $X$, and we have the explicit descriptions of the functors 
$G_Z = \Hom_R(I, -)$ and $F_Z = - \otimes_R I$ as usual.  However, closed subcategories of $\rGr R$ need not be defined 
by two-sided ideals in this way.  For example, consider any $\mb{N}$-graded algebra $R$ and let $Z$ be the subcategory of $\rGr R$ consisting of $\mb{Z}$-graded modules $M$ with $M = M_{\geq 0}$,  in other words, the subcategory of nonnegatively graded modules.  It is clear that $Z$ is closed under subquotients, direct sums, and products, so that $Z$ is a closed subcategory.
However, $Z$ is clearly not equal to $\rGr R/I$ for a graded ideal $I$ of $R$, since any such subcategory defined by an ideal is closed under 
shift, while $Z$ is not.  
\end{example}

Sierra studied the group of autoequivalences of the category of graded modules over the Weyl algebra \cite{Si}.  It seems 
interesting to note that the nontrivial autoequivalences in her work arise as functors $F_Z$ as defined in this paper.
\begin{example}
let $k$ be an algebraically closed field of characteristic $0$.
Let $R$ be the first Weyl algebra $R = k \langle x, y \rangle/(yx-xy-1)$, which is $\mb{Z}$-graded with $\deg x = 1, \deg y = -1$.
The simple objects in $\rGr R$ are parametrized by the affine line over $k$ with its integer points doubled.  More specifically, 
there is a simple module $M_{\lambda}$ for each $\lambda \in k \setminus \{ \mb{Z} \}$ and $2$ simple modules 
$X(n), Y(n)$ for each $n \in \mb{Z}$ \cite[Lemma 4.1]{Si}.  Sierra proves that any autoequivalence of $\rGr R$ is determined by its action on the simple modules \cite[Corollary 5.6]{Si}, and for each $n$ she constructs an interesting autoequivalence $F$ that switches $X(n)$ and $Y(n)$ and fixes all other simple modules   \cite[Proposition 5.7]{Si}.  
This $F$ can be defined as follows:  for each graded rank one projective module $P$, $\Hom_{\rGr R}(P, X(n) \oplus Y(n)) = k$; thus $P$ surjects onto exactly one of the modules $X(n), Y(n)$.  Let $F(P)$ be the kernel of this surjection.  This action extends to a unique exact right exact functor $F$ on the whole category, since $\rGr A$ has enough projectives, similarly as in our discussion after Proposition~\ref{prop:functors}.

The graded simple modules $M$ satisfy $\dim_k M_n \leq 1$ for all $n \in \mb{Z}$ and $\Hom_{\rGr R}(M, M) = k$, from which one may easily see that all graded simple modules are tiny in the category $\rGr R$.  Now for fixed $n$ let $Z$ be the full subcategory of $\rGr R$ consisting of all direct sums of $X(n)$ and $Y(n)$, 
which is closed since these simples are tiny.  Let $F_Z$ be the corresponding functor constructed by Proposition~\ref{prop:functors}(1)
with $S = Z$.  Then since every projective is $Z$-self-effacing, we see that for a rank one projective $P$, the object $F_Z(P)$ is the 
smallest subobject $Q$ of $P$ such $P/Q \in Z$.  Since $\Hom_{\rGr R}(P, X(n) \oplus Y(n)) = k$, it is easy to see that 
$F_Z$ is the same as the functor $F$ described above.
\end{example}

Finally, we discuss our main motivating example: noncommutative projective schemes.  
Let $A$ be an $\mb{N}$-graded $k$-algebra which is connected ($A_0 = k$), finitely generated, and noetherian.  Let $\rTors A$ be the full subcategory of $\rGr A$ consisting of modules $M$ such that for all $m \in M$, $m A_{\geq n} = 0$ for some $n \geq 0$.  
Then one may define the quotient category $\rQgr A = \rGr A/\rTors A$.  
This category has the same objects as $\rGr A$, but for $M \in \rGr A$ we write its image in $\rQgr A$ as $\pi M$.  
The morphisms are given by 
\[
\Hom_{\rQgr A}(\pi M, \pi N) = \dirlim\ \Hom_{\rGr A}(M_{\alpha}, N/N_{\beta}),
\]
where the limit ranges over those graded submodules $M_{\alpha} \subseteq M$ such that $M/M_{\alpha} \in \rTors A$ and 
those graded submodules $N_{\beta} \subseteq N$ such that $N_{\beta} \in \rTors A$.  
 The functor $\pi: \rGr A \to \rQgr A$ is exact, 
and has a left adjoint  $\omega: \rQgr A \to \rGr A$ which can be given explicitly as $\omega(\pi M) = \underset{n \to \infty}{\lim} \bigoplus_m \Hom_{\rGr A}(A_{\geq n}, M(m))$.  The category $\rQgr A$ is called a \emph{noncommutative projective scheme}.    For more details, see \cite{AZ}.

The algebra $A$ is said to satisfy the \emph{$\chi_i$ condition} if  $\dim_k \Ext^j_A(k, M) < \infty$ for all $j \leq i$, for all noetherian modules $M \in \rGr A$, where $k = A/A_{\geq 1}$ is the trivial module.  Most well-behaved graded algebras satisfy $\chi_i$ for all $i \geq 0$, in which case we say that $A$ satisfies $\chi$.  In this case the category $\rQgr A$ is \emph{Ext-finite} in the sense that $\dim_k \Ext^i_{\rQgr A}(M,N) < \infty$ for all noetherian objects $M,N \in \rQgr A$  \cite[Corollary 7.3(3)]{AZ}.

\begin{example}
Let $X = \rQgr A$ for a connected, finitely generated $\mb{N}$-graded noetherian $k$-algebra $A$ satisfying $\chi$, and maintain the notation above.   Since $\rGr A$ is a Grothendieck category, so is its quotient category $X = \rQgr A$ \cite[Corollary 4.6.2]{Po}.   As in the case of categories of quasi-coherent sheaves on commutative schemes, the category $X$ usually does not have enough projectives or exact products.

Since $\{ A(n) | n \in \mb{Z} \}$ generates $\rGr A$,  $\{ \pi A(n) | n \in \mb{Z} \}$ generates $\rQgr A$, and so $X$ is locally noetherian.  Write $\mc{O} = \pi A$.  Artin and Zhang defined a cohomology theory for $X$ as follows: 
for $N \in X$, let $H^i(X, N) = \Ext^i_X(\mc{O}, N)$.  The shift autoequivalence $(1)$ of $\rGr A$ induces an autoequivalence of $\rQgr A$ we also 
write as $(1)$, so $\mc{O}(n) = \pi A(n)$.   Then $H^i(X, N(n)) = \Ext^i_X(\mc{O}, N(n)) = \Ext^i_X(\mc{O}(-n), N)$.

Suppose that $I$ is a graded ideal of $A$.  Then $Z = \rQgr A/I$ is a closed subcategory of $X = \rQgr A$.  The non-trivial proof was given by Smith \cite[Theorem 3.2]{Sm2}, \cite[Theorem 1.2]{Sm3}.  In this case, Artin and Zhang showed that cohomology restricts nicely to such a closed subcategory.  Let $\mc{O}_Z = \pi(A/I)$.  Then if $N \in Z$, we have $H^i(Z,  N) = \Ext^i_Z(\mc{O}_Z, N) = \Ext^i_X(\mc{O}_X, N) = H^i(X, N)$
\cite[Theorem 8.3(3)]{AZ}, similarly as in the commutative case (cf. \cite[Lemma 2.10]{Ha}).  
Thus if $E$ is an injective object in the category $Z$, since  
$Z$ is closed under the shift autoequivalence $(1)$, $E(-n)$ is also injective in $Z$, and 
we will have $\Ext^i_X(\mc{O}(n), E) = \Ext^i_X(\mc{O}_X, E(-n)) = \Ext^i_Z(\mc{O}_Z, E(-n)) = 0$. 
Hence $\{\mc{O}(n) | n \in \mb{Z} \}$ is a set of $S$-projective self-effacing generators, for $S$ the class of injective objects in $Z$.
In particular, $X$ has $S$-projective effacements by Proposition~\ref{prop:self-effacing}, and $Z$ is well-closed in $X$.  As for categories of commutative sheaves, $Z$ will not generally have exact products, and so $Z$ is not generally very well-closed.

However, similarly as for the category $\rGr A$, closed subcategories $Z$ of $\rQgr A$ need not be defined by two sided ideals of $A$, even 
when $Z$ is a closed point.    For example, suppose that $A$ is generated by $A_1$ as a $k$-algebra.  
A \emph{point module} for $A$ is a graded right module $M$ which is generated in degree $0$ and satisfies $\dim_k M_n = 1$ for all $n \geq 0$.   For any point module $M$, $\pi(M)$ is a tiny simple object in $\rQgr A$ \cite[Proposition 5.8]{Sm1}, and thus the category of small direct sums of $\pi(M)$ is a closed point $Z$ in $\rQgr A$.  But generally $M \cong A/I$ for a right but not 2-sided ideal $I$ of $A$, and so $Z$ is not typically defined by a 2-sided ideal of $A$.
\end{example}

To close the paper, we show that our theory applies to prove that closed points, and the locally finite categories built out of them using the Gabriel product, are very well-closed in quite general circumstances.
\begin{theorem}
\label{thm:pt}  $X$ be a locally noetherian Grothendieck $k$-category.
\begin{enumerate}
\item Suppose that $Z$ is a closed point in $X$, that is, $Z$ is the category of direct sums of a tiny simple object $P$ in $X$.
Suppose that $\dim_k \End_X(P) < \infty$ and that $\dim_k \Ext^1_X(M,P) < \infty$ for all noetherian objects $M \in X$. 
Then $X$ has $Z$-projective effacements; that is, $Z$ is very well-closed in $X$.  
\item  Let $Z_1, Z_2, \dots Z_n$ be closed points in $X$ (possibly with repeats), each of which satisfies the hypothesis of (1).  
Then any closed subcategory $Z$ contained in the Gabriel product $Z_1 \cdot Z_2 \cdot \ldots \cdot Z_n$ has $Z$-projective effacements; i.e. 
$Z$ is very well-closed in $X$.
\end{enumerate}
\end{theorem}
\begin{proof}
(1) We would like to apply Proposition~\ref{prop:efface} with $T = \{ P \}$.  For this, as noted in Remark~\ref{rem:single}, it is enough to show that 
for a noetherian object $M$, then \eqref{eq:extprod} holds, namely 
\begin{equation}
\label{eq:extprod2}
\phi: \Ext^1_X(M, \prod_{v \in \beta} N_v) \to \prod_{v \in \beta} \Ext^1_X(M, N_v)
\end{equation}
is an isomorphism, where $\beta$ is a $k$-basis of $\Ext^1_X(M,P)$.  

By assumption, $\Ext^1_X(M,P)$ is finite dimensional over $k$, so $\beta$ is finite.   Thus \eqref{eq:extprod2} holds 
because finite direct products and direct sums coincide, and  $\Ext^1$ commutes with direct sums in the second coordinate when $M$ is noetherian, by Lemma~\ref{lem:extsumprod}(2).  Thus the proof of Proposition~\ref{prop:efface} goes through to show that any noetherian 
object $M$ has a $Z$-projective effacement, since every object in $Z$ is a direct sum of copies of objects in $T$.
Then $X$ has $Z$-projective effacements by Lemma~\ref{lem:pe-props}(3).  

(2).  This is immediate from part (1) and the fact that being very well-closed is stable under the Gabriel product 
(by Lemma~\ref{lem:prodvwc}) and under taking subcategories.
\end{proof}

\begin{corollary}
Let $A$ be a connected $\mb{N}$-graded finitely generated noetherian $k$-algebra satisfying $\chi$, and let $X = \rQgr A$.  Then 
any closed subcategory of a finite Gabriel product of closed points is very well-closed in $X$.
\end{corollary}
\begin{proof}
The $\chi$ condition implies that $X$ is Ext-finite, as already noted.  Since a simple object $P$ is noetherian, $\dim_k \End_X(P) < \infty$ and $\dim_k \Ext^1_X(M, P)< \infty$ holds for all noetherian objects $M$, so the hypothesis of  Theorem~\ref{thm:pt} is satisfied.
\end{proof}

The property of Ext-finiteness should be thought of as a kind of properness assumption for noncommutative schemes.  Thus Theorem~\ref{thm:pt} can be interpreted to say that closed points (or more generally, finite length subschemes) of proper noncommutative schemes ought to be well-closed.   This may be useful for iterating Van den Bergh's blowing up procedure, which does not seem to preserve projectivity in general.

\begin{remark}
When Theorem~\ref{thm:pt}(1) applies, Proposition~\ref{prop:efface} constructs a $Z$-projective effacement of a noetherian object $M$ quite explicitly, by adding $n = \dim_k \Ext^1(M,P)$ copies of $P$ to the bottom of $M$ in the $n$ different ways given by a basis of $\Ext^1_X(M,P)$.
More exactly, the effacement is given by
the exact sequence
\[
0 \to \Ext^1(M,P) \otimes_k P \to \underline{M} \to M \to 0,
\]
corresponding under the correspondence between $\Ext^1$ and extensions to a diagonal element 
\[
\theta = \sum_{i=1}^n v_i \otimes v_i \in  \Ext^1(M,P) \otimes_k \Ext^1(M,P) \cong \Ext^1(M, \Ext^1(M,P) \otimes_k P)
\]
 where $\{ v_i \}$ is a $k$-basis of $\Ext^1(M,P)$.
\end{remark}

\providecommand{\bysame}{\leavevmode\hbox to3em{\hrulefill}\thinspace} \providecommand{\MR}{\relax\ifhmode\unskip\space\fi MR }
\providecommand{\MRhref}[2]{%
  \href{http://www.ams.org/mathscinet-getitem?mr=#1}{#2}
} \providecommand{\href}[2]{#2}

\end{document}